\documentclass[12pt,reqno,a4paper]{amsart}
\usepackage{a4wide}
\usepackage{amsmath,amssymb}
\usepackage{color}
\newtheorem{theor}{Theorem }
\newtheorem{prop}{Proposition}[section]
\newtheorem{theo}[prop]{Theorem}
\newtheorem{lem}[prop]{Lemma}
\newtheorem{coro}{Corollary}

\newtheorem{rem}{Remark}

\newcommand{\D}{\mathbb{D}}
\newcommand{\cA}{{\mathcal{A}}}

\newcommand{\Lat}{\mathrm{Lat}}

\newcommand{\cM}{{\mathcal{M}}}
\renewcommand{\Re}{\mathrm{Re }}

\newcommand{\kH}{\mathrm{H^\infty  }}
\newcommand{\hH}{\mathrm{H^2 }}

\newcommand{\IL}{\mathrm{L^2(\TT)}}

\newcommand{\IIIL}{\mathrm{L^2(\mu)}}

\newcommand{\Ee}{\mathrm{E}}

\newcommand\TT{\mathbb{T}}
\newcommand\DD{\mathbb{D}}

\newcommand\RR{\mathbb{R}}
\newcommand{\FF}{\mathrm{F}}

\DeclareMathOperator{\supp}{supp}

\newcommand{\cZ}{{\mathcal{Z}}}

\newcommand{\cI}{{\mathcal{I}}}
\newcommand{\cD}{{\mathcal{D}}}

\DeclareMathOperator{\dist}{dist}
\numberwithin{equation}{section}

\title{Superharmonically Weighted Dirichlet Spaces}

\author[H. Bahajji-El Idrissi, O. El-Fallah, Y. Elmadani and A. Hanine]{H. Bahajji-El Idrissi, O. El-Fallah, Y. Elmadani and A. Hanine}


\address{Laboratory of Mathematical Analysis and Applications, Mohammed V University in Rabat, B.P. 1014, Rabat, Morocco}
\email{bahajjielidrissihafid@gmail.com}
\email{elfallah@fsr.ac.ma}
\email{elmadanima@gmail.com}
\email{abhanine@gmail.com}

\begin{document}
	\maketitle
	\begin{abstract}
	In this paper, we consider weighted Dirichlet spaces $\cD_\omega$, where $\omega$ is a positive superharmonic weight on the unit disc $\DD$. These spaces include the standard weighted Dirichlet spaces $\cD_\alpha$ and appear in the description of their invariant subspaces. Our goal is to study the spaces $\cD_\omega$. We show that an explicit description of invariant subspaces reduces to the description of those generated by a bounded outer function, and then to the problem of describing cyclic functions, known as the Brown--Shields conjecture. We develop tools, analogous to those used in the harmonic case, that are needed to treat this problem for superharmonically weighted Dirichlet spaces $\cD_\omega$. In particular, we obtain a formula for the Dirichlet integral of outer functions of Carleson--Richter--Sundberg type, estimates for the norm of the reproducing kernel of $\cD_\omega$, and several properties on the capacity associated with $\cD_\omega$.   Using these tools, we provide a description of invariant subspaces when the measure $\Delta \omega$ is finite measure or if the $\supp(\Delta \omega)\cap \TT$ is countable, where $\TT$ denotes the unit circle. Finally, we prove that a smooth outer function  $f \in \cD_\alpha$ such that $\cZ (f) $ is "regular" is cyclic in $\cD_\alpha$ if and only if $c_{\alpha }(\cZ(f))= 0$.
\end{abstract}
\tableofcontents
\section{Introduction}
Weighted Dirichlet spaces are Hilbert spaces formed by holomorphic functions on the unit disc whose derivatives are square--integrable with respect to a weighted measure. 
Numerous unsolved problems remain in the theory of these spaces, including the estimation of the reproducing kernel, understanding the structure of invariant subspaces, addressing the cyclicity problem, and characterizing zero sets and interpolating sequences. These problems are deeply interconnected and continue to invite investigation. Note that some references exist on the classical Dirichlet space (see \cite{Ar, ARSW1, Ros}), as well as the textbook \cite{EKMR14}, which offers a comprehensive study of harmonically weighted Dirichlet spaces. 

Let $\mathbb{D}$ denote the unit disc and $\mathbb{T}$ the unit circle. A positive measurable function $\omega$ on $\mathbb{D}$ is called a weight if it is integrable with respect to the Lebesgue area measure on $\DD$. The weighted Dirichlet space $\mathcal{D}_{\omega}$ is defined by
\[
\mathcal{D}_{\omega}
:= \left\{ f \in \hH :
\cD_\omega(f):=\int_{\mathbb{D}} |f'(z)|^2 \omega(z)\, dA(z) < \infty \right\},
\]
where $\hH$ denotes the Hardy space on \(\mathbb{D}\), and \(dA(z) = \frac{dx \, dy}{\pi}\) is the normalized area measure. When \(\omega\) is a superharmonic function, \(\cD_\omega\) is known as a \textit{superharmonically weighted Dirichlet space}. 

The space \(\cD_\omega\) is endowed with the Hilbert space norm 
\[
\|f\|_{\cD_\omega}^2 = \|f\|_{\hH}^2 + \cD_\omega(f), \quad f \in \cD_\omega.
\]
In the special case $\omega \equiv 1$, $\mathcal{D}_\omega$ reduces to $\mathcal{D}$, the classical Dirichlet space. 
For \( \omega_\alpha(z) := (1 - |z|^2)^\alpha\) with \(z \in \mathbb{D}\) and \(\alpha \in (0, 1)\), the space \(\cD_{\omega_\alpha}\) coincides with 
the standard weighted Dirichlet space \(\cD_\alpha\). In particular, \(\cD \subset \cD_\alpha\) for each \(\alpha \in (0,1)\). More generally, \(\cD\) is contained in \(\cD_\omega\) if and only if \(\liminf_{|z| \to 1^-} \omega(z) = 0\); see \cite{BE_0} for further details. 
Harmonically weighted Dirichlet spaces were introduced by Richter \cite{Ri1} to provide a description of the invariant subspaces of $\cD$ analogous to Beurling's theorem for the Hardy space. 
In contrast, the analogous description of the invariant subspaces of $\mathcal{D}_\alpha$ involves superharmonically weighted Dirichlet spaces, as shown by Aleman \cite{Al}. Consequently an explicit description of the invariant subspaces of $\cD_\alpha$ requires a detailed analysis of such superharmonically weighted Dirichlet spaces.

A closed subspace $\cM$ of $\cD_\omega$ is called an invariant subspace
(or $S_\omega$-invariant subspace, also referred to as a shift--invariant subspace) if $S_\omega(\cM) \subset \cM$, where $S_\omega$ denotes the multiplication operator by $z$, defined by
\begin{align*}\label{shift}
	S_\omega \colon \cD_\omega \longrightarrow \cD_\omega, \quad f \mapsto S_\omega(f)(z) = z f(z).
\end{align*} 
%
Note that $S_\omega$ is a bounded operator on $\mathcal{D}_\omega$. The collection of all closed $S_\omega$-invariant subspaces of $\mathcal{D}_\omega$ will be denoted by $\mathrm{Lat}(S_\omega)$. The structure of shift--invariant subspaces of harmonically weighted Dirichlet spaces was established by Richter and Sundberg \cite{RS2}.
%
Their approach is primarily based on two key observations. The first is that every shift--invariant subspace is generated by an extremal function. 
The second concerns an estimate of the local Dirichlet integral with respect to $\TT$ for cut--off functions. Aleman \cite{Al} extended the first result to the superharmonically weighted Dirichlet space. In Section \ref{Cut--off functions}, we show that the second observation remains valid with respect to $\DD$. We therefore obtain the following theorem.

Let $\cM$ be a nontrivial invariant subspace of $\cD_\omega$, and let $\mathcal{I}_{\cM}$ denote its greatest common inner divisor. For $f \in \cD_\omega$, we write $[f]_{\cD_\omega}$ for the smallest closed invariant subspace of $\cD_\omega$ containing $f$.

\begin{theor}\label{RSinvar}  Let $\omega$ be a positive superharmonic weight on $\DD$, and   let $\cM\in \Lat(S_\omega)\setminus\{0\}$. 
	Then there exists an outer function  $f\in   \cD_\omega\cap \kH$ such that 
	\[
	\cM=\mathcal{I}_\cM \hH\cap [f]_{  \cD_\omega}.
	\]
	Furthermore $f$ can be chosen so that $f$ and $\mathcal{I}_\cM f$ are multipliers of $\cD_\omega$.
\end{theor}
It is worth emphasizing that this result is new even in the setting of the standard weighted Dirichlet space \(\cD_\alpha\), where \(\alpha \in (0,1)\).\\

Thanks to Theorem \ref{RSinvar}, to obtain an explicit description of invariant subspaces of $\cD_\omega$ it suffices to describe  $[f]_{  \cD_\omega}$ when $f$ is an outer function.
%
%
Recall that  \( f \in \cD_\omega \) is called \textit{cyclic} (or \textit{cyclic} in \(\cD_\omega\)) with respect to the shift operator \( S_\omega \) if \( [f]_{\cD_\omega} = \cD_\omega \). Understanding the cyclic functions in \(\cD_\omega\) is a significant step toward a description of the shift--invariant subspaces of \(\cD_\omega\). 


A necessary condition for cyclicity follows directly from a weak--type inequality for the capacity \(c_\omega\) associated with \(\cD_\omega\) (see Section \ref{capomega} for the definition of \(c_\omega\)); namely, if \(f\) is cyclic in \(\cD_\omega\), then \(f\) is outer and \(c_{\omega}(\cZ(f)) = 0\), where  
\[
\cZ(f) := \left\{ \zeta \in \TT : \lim_{r \to 1^-} f(r\zeta) = 0 \right\}.
\]  
In 1984, Brown and Shields \cite{BS} conjectured that an outer function \(f \in \cD\) is cyclic if and only if 
\[
c_m(\mathcal{Z}(f))=0,
\]
where \(m\) denotes the Lebesgue measure on \(\TT\). The (extended) Brown--Shields conjecture asserts the following.\\

\textbf{Brown--Shields conjecture.} Let $\omega$ be a positive superharmonic weight on $\DD$, and let $f \in \cD_\omega$. Then the following are equivalent:
\[
f \text{ is cyclic in } \cD_\omega 
\quad \Longleftrightarrow \quad 
f \text{ is outer and } c_\omega(\mathcal{Z}(f)) = 0.
\]

It is worth noting that Theorem \ref{equivalence} shows in particular that, if the Brown--Shields conjecture holds for all positive superharmonic weights, then every invariant subspace $\cM$ of $\cD_\omega$ admits the representation
\[
\cM = \cI_\cM \hH \cap \left\{ f \in \cD_\omega : f = 0 \ \text{on } \Ee \ \text{$c_\omega$-q.e.} \right\},
\]
where $\Ee$ is a Borel subset of $\TT$. For a precise statement, see the end of Section \ref{Structure}.\\



Recall that every positive superharmonic function $\omega$ admits a Riesz decomposition and can be represented as the sum of a Green potential $G_{\mu}$ and a Poisson integral $P_{\nu}$. More precisely, for $z\in \mathbb{D}$,
\[
\omega(z) = G_{\mu}(z) + P_{\nu}(z),
\]
where
\[
G_{\mu}(z) = \int_{\mathbb{D}} \log \left| \frac{1 - \overline{w} z}{z - w} \right| \, d\mu(w),
\]
and
\[
P_{\nu}(z) = \int_{\mathbb{T}} \frac{1 - |z|^2}{|\zeta - z|^2} \, d\nu(\zeta).
\]
Here, \(\mu\) (respectively, \(\nu\)) is a positive Borel measure on \(\mathbb{D}\) (respectively, a finite positive Borel measure on \(\mathbb{T}\)), referred to as the Riesz measure (respectively, the boundary measure) of \(\omega\). Note that \(G_\mu\) is also known as the pure superharmonic component of \(\omega\), while \(P_\nu\) represents the harmonic component. We denote by $\cD_\mu$ (respectively, $\cD_\nu$) the Dirichlet space corresponding to the case $\omega = G_\mu$ (respectively, $\omega = P_\nu$).
The function \(\omega\) is not identically \(-\infty\) if and only if its Riesz measure \(\mu\) satisfies the condition
\begin{equation} \label{RMF}
	\int_{\mathbb{D}} (1 - |w|^2) \, d\mu(w) < \infty.
\end{equation}	Throughout this paper, we assume that condition \eqref{RMF} holds. \\




\noindent For \(0<\alpha<1\), the standard weighted Dirichlet space \(\cD_{\alpha}\) corresponds to the choice of weight \(\omega = G_{\mu_\alpha}\), where
\[
d\mu_{\alpha}(z) := -\Delta(1-|z|^2)^{1-\alpha}\, dA(z).
\]
It was shown by Sarason \cite{Sar} (see also \cite{Gui}) that, when
\(\omega=P_{\delta_{\zeta}}\), where \(\delta_{\zeta}\) denotes the Dirac measure
at \(\zeta\in\mathbb{T}\), the Brown--Shields conjecture holds for the space
\(\mathcal{D}_{\delta_{\zeta}}\). In 2016, El-Fallah, Elmadani, and Kellay
\cite{EEK} proved that the Brown--Shields conjecture is true for
\(\mathcal{D}_{\nu}\) whenever \(\operatorname{supp}(\nu)\) is countable.\\

%
One of the key ingredients in the proof, in the harmonic case, is the Richter--Sundberg formula, which expresses the Dirichlet integral of an outer function $f$ with respect to $\nu$ in terms of the modulus of its radial limits $|f^*|$, where
\[
f^*(e^{it}) := \lim_{r \to 1^-} f(re^{it}).
\]
For the superharmonically weighted space $\cD_\mu$, we establish an analogue of this result. This identity allows us to obtain a precise estimate for the norm of the reproducing kernel of $\cD_\mu$, which in turn leads to a characterization of those $\zeta \in \TT$ for which $c_{\mu}(\{\zeta\})=0$.

In section \ref{Structure}, we prove Theorem \ref{RSinvar} and examine the behavior of the lattice of invariant subspaces with respect to inner--outer factorization, as well as operations involving cut--off functions. 
In the final section, we exploit the results developed in the previous sections to establish explicit criteria for cyclicity. 

We first consider the case where $\mu$ is a finite measure. This setting is closely related to the Hardy space framework, since the $c_\mu$-polar sets are precisely the sets of Lebesgue measure zero. In this case, the Brown--Shields conjecture has a positive answer.

The main result of Subsection~\ref{Countable} yields an affirmative answer to the Brown--Shields conjecture for $\mathcal{D}_\mu$ in the case where $\operatorname{supp}(\mu)\cap\mathbb{T}$ is countable. We state the precise formulation in the following theorem.
For \( f \in \mathcal{D}_\mu \), we define
\[
\underline{\mathcal{Z}}(f)
:= \left\{ \zeta \in \overline{\mathbb{D}} : \liminf_{z \to \zeta} |f(z)| = 0 \right\}.
\]
\begin{theor}\label{theo61}
	Let $\mu$ be a positive Borel measure on $\DD$, and let
	$f\in\mathcal D_\mu$ be an outer function such that
	$\operatorname{supp}(\mu)\cap \underline{\mathcal Z}(f)$ is countable.
	Then $f$ is cyclic in $\mathcal D_\mu$ if and only if
	\(
	c_\mu\bigl(\mathcal Z(f)\bigr)=0.
	\)
\end{theor}
\begin{coro}
	%
	Let $\mu$ be a positive Borel measure on $\mathbb{D}$. The Brown--Shields conjecture holds for $\mathcal{D}_\mu$ whenever $\operatorname{supp}(\mu)\cap \mathbb{T}$ is countable.
\end{coro}

\begin{coro}
	Let \(\mu\) be a positive Borel measure on \(\DD\), and assume that \(\supp(\mu) \cap \TT\) is countable. Then for every \(\cM\in \Lat(S_\omega)\setminus\{0\}\), there exists a Borel subset \(\Ee\) of \(\supp(\mu)\) such that 
	$$
	\cM = \cI_\cM \hH \cap \{ g \in \mathcal{D}_\mu : g = 0 \ \text{on } \Ee \ \text{$c_\mu$-q.e.} \}.
	$$
\end{coro}
We conclude the paper by considering the standard weighted Dirichlet space $\mathcal{D}_\alpha$. We show that an outer function $f \in A(\mathbb{D}) \cap \mathcal{D}_\alpha$, whose zero set $\mathcal{Z}(f)$ is a ``regular'' subset of $\mathbb{T}$, is cyclic in $\mathcal{D}_\alpha$ if and only if $c_{\alpha}(\mathcal{Z}(f)) = 0$. Here, $A(\mathbb{D})$ denotes the disc algebra.

\section{The Dirichlet Integral}

\subsection{Harmonic weights} Let $\omega = P_\nu$, where $\nu$ is a
finite positive Borel measure on $\TT$, be a harmonic weight on $\DD$. For a given function $f \in \cD_\nu$, an application of Fubini's theorem yields 
\[
\begin{array}{lll}
	\cD_\nu (f)& =& \displaystyle \int_\DD|f'(z)|^2P_\nu(z)dA(z)\\
	& = &\displaystyle \int_\TT  \left ( \displaystyle \int_\DD|f'(z)|^2P_{\delta_\zeta}(z)dA(z)\right )d\nu(\zeta)\\
	& := & \displaystyle \int_\TT \cD_\zeta (f)d\nu(\zeta),
\end{array}
\]
where $\cD_\zeta(f)$ is the \textit{local Dirichlet integral} of a function \( f \in \hH \) at a point \(\zeta \in \mathbb{T}\) given by 
$$
\cD_\zeta (f)= \displaystyle \int_\DD|f'(z)|^2P_{\delta_\zeta}(z)dA(z).
$$
It is known that if $\cD_\zeta (f)$ is finite, then the radial limit of $f$ at $\zeta$, also denoted by $f(\zeta)$, exists. In this case, Douglas' formula gives the following alternative expression
\[
\cD_{\zeta}(f) = \int_{\mathbb{T}} \frac{|f(\zeta) - f(\lambda)|^2}{|\zeta - \lambda|^2} \, dm(\lambda).
\]
Further details may be found in \cite{EKMR14}. This formula has been extended to various spaces of analytic functions; see \cite{BE25} and the references therein.\\
Recall that an \textit{outer function} is an element \( f \in \mathrm{H}^1 \) that can be represented in the form
\[
f(z) = \alpha \exp\left( \int_{\mathbb{T}} \frac{e^{i\theta} + z}{e^{i\theta} - z} h(e^{i\theta}) \, dm(e^{i\theta}) \right), \quad z \in \mathbb{D},
\]
where \( h \) is a real-valued integrable function on $\TT$ and \( |\alpha| = 1 \). Moreover, \( h(e^{i\theta}) = \log |f(e^{i\theta})| \) for \( m \)-almost every \( e^{i\theta} \in \TT\) (cf. \cite{Koo}).\\ 
In \cite{RS1}, Richter and Sundberg showed that if \( f\in\hH \) is an outer function, then
\[
\cD_{\zeta}(f) = \int_{\mathbb{T}} \frac{|f(\lambda)|^2-|f(\zeta)|^2- |f(\zeta)|^2\log(|f(\lambda)|^2/|f(\zeta)|^2)}{|\zeta - \lambda|^2} \, dm(\lambda).
\]
Define the function \( \FF : \mathbb{R} \times (\mathbb{R} \cup \{-\infty\}) \to \mathbb{R} \) by
\[
\FF(x,y) =
\begin{cases}
	e^{x} - e^{y} - e^{y}(x-y), & x,y \in \mathbb{R}, \\
	e^{x}, & x \in \mathbb{R},\ y = -\infty.
\end{cases}
\]
For almost every \( \zeta \in \mathbb{T} \), the above identity can be rewritten as
\[
\mathcal{D}_{\zeta}(f)
=
\int_{\mathbb{T}}
\frac{F\big(2h(\lambda),\,2h(\zeta)\big)}{|\zeta - \lambda|^{2}} \, dm(\lambda),
\]
where \( h(e^{it}) = \log |f(e^{it})| \).\\
Observe that \(\FF\) coincides with the \textit{Bregman divergence} associated with the strictly convex function \(\Phi(t)=e^t\). Recall that, for a strictly convex function \(\Phi:\RR\to\RR\), the corresponding Bregman divergence is defined by
\[
\mathrm{D}_{\Phi}(x,y)
=
\Phi(x)-\Phi(y)-\Phi'(y)(x-y),
\]
where $x,y\in\RR$; see, e.g., \cite{Br}.
\subsection{Pure superharmonic weights}
Let $\omega = G_\mu$ be the superharmonic weight associated with the positive Borel measure $\mu $ on $\DD$.  Analogously, the local Dirichlet integral of a function \( f \in \hH \) at a point \( w \in \mathbb{D} \) is defined by
\[
\cD_w(f) = \int_{\mathbb{T}} \frac{|f(\zeta)-f(w)|^2}{|\zeta - w|^2} \, dm(\zeta)=\int_{\mathbb{T}} \frac{|f(\zeta)|^2 - |f(w)|^2}{|\zeta - w|^2} \, dm(\zeta).
\]
Furthermore, by using the Littlewood-Paley formula, the Dirichlet integral {$\cD_\mu$} is given by
\[
\cD_\mu(f) = \int_{\mathbb{D}} \cD_w(f) \, (1-|w|^2)d\mu(w).
\]
Aleman \cite{A2} showed that if \( f \) is an outer function, then for every $w\in\DD$, the local Dirichlet integral \(\cD_w(f)\) admits the representation
\begin{align}\label{Efunction}
	(1 - |w|^2) \cD_w(f) &= \int_{\mathbb{T}} |f(\zeta)|^2 \, d\sigma_w(\zeta) - \exp\left( \int_{\mathbb{T}} 2\log |f(\zeta)| \, d\sigma_w(\zeta) \right)\nonumber\\
	&:=\mathcal{E}_{\sigma_w}(|f|^2),
\end{align}
where \[ d\sigma_w(\zeta) := \frac{1 - |w|^2}{|\zeta - w|^2} \, dm(\zeta) .\]
To represent this local Dirichlet integral in terms of the function \( \FF \), we introduce the concept of a \(\Phi\)-entropy; see, e.g., \cite{Chaf}. Let \( \Phi: I \to \mathbb{R} \) be a smooth convex function on a closed interval \( I \subset \mathbb{R} \), which may be unbounded. 
Let \( \sigma \) be a probability measure on a Borel space \( (\Omega, \mathcal{F}) \). For any \( \sigma \)-integrable function \( f : (\Omega, \mathcal{F}) \to (I, \mathcal{B}(I)) \), the \( \Phi \)-entropy functional \( \mathrm{Ent}_\Phi^\sigma \) is defined by
\[
\text{Ent}_\Phi^\sigma(f) := \int_\Omega \Phi(f) \, d\sigma - \Phi\left(\int_\Omega f \, d\sigma\right) := \mathbb{E}_\sigma(\Phi(f)) - \Phi(\mathbb{E}_\sigma(f)).
\]

\begin{prop}\label{Entropie}
	Let \( \sigma \) be a probability measure on \( (\Omega, \mathcal{F}) \), and let \( \Phi \) be a stricly convex function on an interval \( I \subset \mathbb{R} \). Then
	\[
	\mathrm{Ent}_\Phi^\sigma(f) = \inf_{a \in \mathbb{R}} \mathbb{E}_\sigma\left(\mathrm{D}_\Phi(f,a)\right)
	\]
	for any \( \sigma \)-integrable function \( f: (\Omega, \mathcal{F}) \to (I, \mathcal{B}(I)) \).
\end{prop}
\begin{proof}
	Let \(\psi(a) = \mathbb{E}_\sigma(\Phi(f) - \Phi(a) - \Phi'(a)(f - a))\), where \(a \in \mathbb{R}\). Differentiating with respect to the variable \(a\), we find
	\[
	\frac{\partial \psi(a)}{\partial a} = \mathbb{E}_\sigma(-\Phi'(a) + \Phi'(a) - \Phi''(a)(f - a)) = -\Phi''(a)(\mathbb{E}_\sigma(f) - a).
	\]
	The minimum is achieved at \(a = \mathbb{E}_\sigma(f)\). Then \begin{equation}\label{equa2.2}
		\arg \min_{a\in\RR} \mathbb{E}_\sigma(\mathrm{D}_\Phi(f,a))=\mathbb{E}_\sigma(f)
	\end{equation} Hence the result follows.
\end{proof}
The identity \eqref{equa2.2} reflects a general optimality property of Bregman divergences and their connection with entropy--type functionals. For a systematic study of the relationship between Bregman divergences and entropy, as well as a characterization of the converse implication in \eqref{equa2.2}, we refer the reader to \cite{Banerjee2005}.

Now, we deduce the connection between the Dirichlet integral associated with \( \mu \) and the function \( \FF \).
\begin{theo}\label{Function}
	Let \( f \in \hH \) be an outer function. Then
	\[
	\cD_\mu(f) = \int_{w\in\DD}\inf_{a \in \mathbb{R}} \int_{\zeta \in \TT} \FF(2\log|f(\zeta)|, a) d\sigma_w(\zeta)d\mu(w).
	\]
\end{theo}
\begin{proof}
	Let \( w \in \mathbb{D} \). By \eqref{Efunction} and Proposition~\ref{Entropie}, we obtain
	\[
	\begin{aligned}
		(1 - |w|^2)\mathcal{D}_w(f)
		&= \mathrm{Ent}_{\exp}^{\sigma_w}(2 \log |f|) \\
		&= \inf_{a \in \mathbb{R}} \mathbb{E}_{\sigma_w}\!\left(\mathrm{D}_{\exp}(2\log|f|,a)\right) \\
		&= \inf_{a \in \mathbb{R}} \mathbb{E}_{\sigma_w}\!\big(\FF(2\log |f|, a)\big) \\
		&= \inf_{a \in \mathbb{R}} \int_{\mathbb{T}} \FF\big(2 \log |f(\zeta)|, a\big)\, d\sigma_w(\zeta).
	\end{aligned}
	\]
	Integrating over \( \mathbb{D} \) with respect to \( \mu \) yields the result.
\end{proof}
\subsection{Relation between harmonic weights and purely superharmonic weights}
%
In this subsection, we recall a result of Shimorin, presented on pages 291--292 of \cite{Shi2}, in which the Dirichlet integral \( \mathcal{D}_\nu(f) \) is expressed in terms of \( \mathcal{E}_{\sigma_w}(|f|^2) \). This naturally leads us to consider the family of functions

\[
V_r(z) = \int_{\mathbb{T}} \frac{r^2(1 - |z|^2)}{|\zeta - rz|^2} \, d\nu(\zeta), \quad z\in\DD,\,\, r \in (0, 1).
\]
Let \( r \in (0,1) \). The function \( V_r \) is a positive superharmonic function on \( \mathbb{D} \). Moreover, \( V_r \) admits the following representation as a logarithmic potential of a positive measure \( \nu_r \) on \( \mathbb{D} \):
\[
V_r(z) = \int_{\mathbb{D}} \log \left| \frac{1 - z \overline{w}}{z - w} \right|^2 \, d\nu_r(w), \qquad z \in \mathbb{D},
\]
where
\[
d\nu_r(w) = -\Delta V_r(w)\, dA(w),
\]
and \( \Delta \) denotes the Laplacian on \( \mathbb{D} \).\\
For each fixed \( z \in \mathbb{D} \), the mapping \( r \mapsto V_r(z) \) is decreasing on \( (0,1) \). Furthermore, as \( r \to 1 \), \( V_r(z) \) converges pointwise to the Poisson integral of \( \nu \), that is,
\[
\lim_{r \to 1} V_r(z) = {P_\nu}(z).
\]
The following theorem establishes the desired connection.
\begin{theo}[Shimorin \cite{Shi2}]\label{Shimorin}
	Let $\nu$ be a finite positive Borel measure on $\mathbb{T}$, and let \( f \in \mathcal{D}_\nu \). Then
	\[
	\mathcal{D}_\nu(f)
	= \lim_{r \to 1} \int_{\mathbb{D}} \mathcal{E}_{\sigma_w}(f)\, d\nu_r(w).
	\]
\end{theo}

\subsection{Dirichlet integral and Cut--off functions}\label{Cut--off functions}

Let $f$ and $g$ be two outer functions in $\hH$. The cut--off functions $f\wedge g$ and $f\vee g$ are the outer functions defined by 
$$
f\wedge g(z)=\exp\left(\int_{\TT}\frac{e^{i\theta}+z}{e^{i\theta}-z}\log\left(\min(|f(e^{i\theta})|,|g(e^{i\theta})|)\right)dm(e^{i\theta})\right),\quad z\in\DD, 
$$
and 
$$
f\vee g(z)=\exp\left(\int_{\TT}\frac{e^{i\theta}+z}{e^{i\theta}-z}\log\left(\max(|f(e^{i\theta})|,|g(e^{i\theta})|)\right)dm(e^{i\theta})\right),\quad z\in\DD.
$$
An important problem is to determine conditions under which these cut--off functions belong to the corresponding Dirichlet spaces. The following theorem addresses this issue in the setting of superharmonic weights.
\begin{theo}\label{cut--off}
	Let \(\omega\) be a positive superharmonic weight on \(\DD\). Let \(f\) and \(g\) be outer functions in \(\hH\). The following statements hold
	\begin{itemize}
		\item[A.] $\mathcal{D}_\omega(f\vee g)\leq \cD_\omega(f)+\cD_\omega(g) $,
		\item[B.] $\cD_\omega(f\wedge g)\leq \cD_\omega(f)+\cD_\omega(g) $,
		\item[C.]$\cD_\omega(f\wedge f^2)\leq 4\cD_\omega(f).$
	\end{itemize}
\end{theo}
%
The proof of Theorem \ref{cut--off} is carried out by two distinct approaches. The first establishes the result in the harmonic setting and subsequently extends it to superharmonic functions. The second proceeds in the reverse direction, treating the superharmonic case directly and deriving the harmonic case as a consequence. This yields two independent proofs of the theorem.

\subsubsection{First Approach}
{Let \( a, b \in \mathbb{R} \), we denote} \( a \wedge b := \min(a, b) \) and \( a \vee b := \max(a, b) \). Based on certain min-max properties of the function \(\FF\), Richter and Sundberg established the following results.

\begin{lem}\label{min-max}
	Let \( x_1, x_2 \in \mathbb{R} \) and \( y_1, y_2 \in \mathbb{R} \cup \{\infty\} \). Then 
	\[
	\FF(x_1 \wedge x_2, y_1 \wedge y_2) \leq \max\left(\FF(x_1, y_1), \FF(x_2, y_2)\right),
	\]
	and
	\[
	\FF(x_1 \vee x_2, y_1 \vee y_2) \leq \max\left(\FF(x_1, y_1), \FF(x_2, y_2)\right).
	\]
\end{lem}

\begin{lem}\label{RS2}
	Let \( g \) be the real-valued function defined by
	\[
	g(x) =
	\begin{cases} 
		x, & \text{if } x \geq 0, \\ 
		2x, & \text{if } x \leq 0.
	\end{cases}
	\]
	Then
	\[
	\FF(g(x), g(y)) \leq 4 \FF(x, y), \quad \forall x, y \in \mathbb{R}.
	\]
\end{lem}

\begin{proof}[Proof of Theorem \ref{cut--off}]
	Building on Lemma \ref{min-max} and Lemma \ref{RS2}, Richter and Sundberg proved Theorem \ref{cut--off} in the case where \(\omega\) is a harmonic function in \(\mathbb{D}\) (see \cite{RS2}). Using Theorem \ref{Function}, we extend this result to general superharmonic functions.\\
	Let \( f, g \in \cD_\mu \) be outer functions. We compute
	\begin{align*}
		\cD_\mu(f \wedge g) &= \int_{\mathbb{D}} \inf_{a \in \mathbb{R}} \int_{\mathbb{T}} \FF(2\log|f \wedge g|(\zeta), a) \, d\sigma_w(\zeta) \, d\mu(w) \\
		&= \int_{\mathbb{D}} \inf_{a \in \mathbb{R}} \int_{\mathbb{T}} \FF(2\log|f(\zeta)| \wedge 2\log|g(\zeta)|, a) \, d\sigma_w(\zeta) \, d\mu(w) \\
		&\leq \int_{\mathbb{D}} \inf_{a \in \mathbb{R}} \int_{\mathbb{T}} \left[\FF(2\log|f(\zeta)|, a) + \FF(2\log|g(\zeta)|, a)\right] \, d\sigma_w(\zeta) \, d\mu(w) \\
		&\leq \cD_\mu(f) + \cD_\mu(g).
	\end{align*}
	The second assertion follows by the same argument.\\
	To prove \(\mathrm{C.}\), note that \( |f \wedge f^2| = e^{g(u)} \), where \( u = \log|f| \), and the function \( g \) is given by
	\[
	g(x) =
	\begin{cases}
		x, & \text{if } x \geq 0, \\
		2x, & \text{if } x \leq 0.
	\end{cases}
	\]
	It follows from Lemma~\ref{RS2} that
	\begin{align*}
		\cD_\mu(f \wedge f^2) &= \int_{\mathbb{D}} \inf_{a \in \mathbb{R}} \int_{\mathbb{T}} \FF(2\log|f \wedge f^2|(\zeta), a) \, d\sigma_w(\zeta)\, d\mu(w) \\
		&= \int_{\mathbb{D}} \inf_{a \in \mathbb{R}} \int_{\mathbb{T}} \FF(g(u(\zeta)), g(a)) \, d\sigma_w(\zeta) \, d\mu(w) \\
		&\leq 4 \int_{\mathbb{D}} \inf_{a \in \mathbb{R}} \int_{\mathbb{T}} \FF(u(\zeta), a) \, d\sigma_w(\zeta) \, d\mu(w) \\
		&= 4 \cD_\mu(f).
	\end{align*}
	This completes the proof of Theorem \ref{cut--off}.
\end{proof}
\subsubsection{Second Approache}
To provide a new proof of Theorem \ref{cut--off}, independent of the method of Richter and Sundberg, we require the following lemma.
\begin{lem}\label{LemmaCutoff}
	Let \((\Omega,\mathcal{A},\sigma)\) be a probability space, and let \(f,g \in \mathrm{L}^2(\sigma)\) satisfy \(\sigma(\{|fg|=0\})=0\). Then
	\begin{enumerate}
		\item[A.] $\mathcal{E}(f\wedge g)\leq \mathcal{E}(f)+\mathcal{E}(g) $,
		\item[B.] $\mathcal{E}(f\vee g)\leq \mathcal{E}(f)+\mathcal{E}(g) $,
		\item[C.]  $\mathcal{E}(f\wedge f^2)\leq 4\mathcal{E}(f)$,
	\end{enumerate}
	where
	\(
	\mathcal{E}(f)
	:= \int_\Omega |f|^2\, d\sigma
	- \exp\!\left(\int_\Omega \log |f|^2\, d\sigma\right).
	\)
\end{lem}
\begin{proof}
	We denote by $A:=\left\lbrace  \vert g \vert \leq \vert f \vert  \right\rbrace$ and $A^c:=\left\lbrace  \vert g \vert > \vert f \vert  \right\rbrace$. Set $\alpha=\sigma(A)$ and $\beta=1-\alpha$. If $\alpha\beta=0$ then A. and B. are immediate. Hence we assume that $\alpha,\beta>0$.\\ 
	We introduce the notation
	\begin{eqnarray*}
		A_f:=\int_A \vert f \vert^2d\sigma \quad &&\mbox{and}\qquad  A_f^c:= \int_{A^c} \vert f \vert^2d\sigma\\
		x_1(f):= \exp\int_{A} \log\vert f \vert^2d\sigma  \quad &&\mbox{and}\quad  x_2(f):= \exp\int_{A^c} \log\vert f \vert^2d\sigma.
	\end{eqnarray*}   
	Consider the function $h$ defined  on $D:=\left\lbrace  0\leq y_1 \leq x_1   \;\mbox{ and }\;  0\leq  x_2\leq y_2\right\rbrace $ by
	$$h(x_1,x_2,y_1,y_2):=x_1x_2+y_1y_2-x_2y_1. $$ 
	Since $\partial_{x_2}h=x_1-y_1\geq 0$ on $D$, we have $h(x_1,x_2,y_1,y_2)
	\leq h(x_1,y_2,y_1,y_2)=x_1y_2$ for all $(x_1,y_2,y_1,y_2)\in D$.
	Moreover,
	\[
	\big(x_1(f), x_2(f), x_1(g), x_2(g)\big) \in D.
	\]
	By Jensen's inequality, 
	$$x_1(f)\leq (A_f/\alpha)^\alpha  \qquad\mbox{ and} \qquad x_2(g) \leq (A_g^c/\beta)^\beta.$$
	Using Young's inequality, we obtain
	\begin{eqnarray}
		&&\exp\int_\TT \log\vert f \vert^2d\sigma+\exp\int_\TT \log\vert g \vert^2d\sigma -\exp\int_\TT \log\vert f \wedge g \vert^2d\sigma  \nonumber\\
		&& = h(x_1(f),x_2(f),x_1(g),x_2(g))  \nonumber\\
		&&\leq  (A_f/\alpha)^\alpha(A_g^c/\beta)^\beta\nonumber\\
		&&\leq A_f+A_g^c \nonumber\\
		&&=\int_\TT \vert f \vert^2d\sigma+\int_\TT \vert g \vert^2d\sigma -\int_\TT \vert f\wedge g \vert^2d\sigma.\nonumber
	\end{eqnarray}
	Hence assertion A. follows. Assertion B. is obtained by the same argument. For C. see \cite[Theorem~3.1]{BE}. This completes the proof.
\end{proof}
\begin{proof}[Second proof of Theorem \ref{cut--off}]
	Let \(w \in \mathbb{D}\). By \eqref{Efunction}, we have  
	\[
	(1-|w|^2)\mathcal{D}_{w}(f) = \int_{\zeta \in \mathbb{T}} |f(\zeta)|^2 d\sigma_w(\zeta) 
	- \exp\left(\int_{\zeta \in \mathbb{T}} \log|f(\zeta)|^2 d\sigma_w(\zeta)\right),
	\] 
	where \(d\sigma_{w}(\zeta) = \frac{1 - |w|^2}{|w - \zeta|^2} dm(\zeta)\). \\
	Applying Lemma \ref{LemmaCutoff} with \((\Omega, \cA, \sigma) = (\mathbb{T}, \mathcal{B} _\TT, \sigma_{w})\), we obtain 
	\begin{enumerate}
		\item[A.] $\cD_{\mu}(f\vee g)\leq \cD_{\mu}(f)+\cD_{\mu}(g),$
		\item[B.] $\cD_{\mu}(f\wedge g) \leq \cD_{\mu}(f)+\cD_{\mu}(g),$
		\item[C.] $\cD_{\mu}(f\wedge f^2)\leq 4 \, \cD_{\mu}(f)$.
	\end{enumerate}
	Combining Lemma \ref{LemmaCutoff} and Theorem \ref{Shimorin}, we deduce
	\begin{enumerate}
		\item[A.] $\cD_{\nu}(f\vee g)\leq \cD_{\nu}(f)+{\cD_{\nu}(g)},$
		\item[B.] $\cD_{\nu}(f\wedge g) \leq \cD_{\nu}(f)+\cD_{\nu}(g),$
		\item[C.] $\cD_{\nu}(f\wedge f^2)\leq 4 \, \cD_{\nu}(f)$.
	\end{enumerate}
\end{proof}
\section{Reproducing kernel of $\cD_\omega$}

Since $\cD_\omega$ is continuously embedded in $\mathrm{H}^2$, $\cD _\omega$ is a reproducing kernel Hilbert space. Its reproducing kernel will be denoted by $K^\omega (.,.)$. Recall that for a superharmonic weight $\omega$, the kernel $K^{\omega}$ of $\cD_\omega$, satisfies the complete Pick-Nevanlinna property (CNP) \cite{Shi2}. For the definition of (CNP) and further properties see \cite{AM, AHMR1, H} .\\

If  $\omega $ is harmonic, it was shown in \cite{EEK18} that 
$$
K^\omega(z, z) \asymp 1+ \displaystyle \int _0^{|z|}\frac{dr}{(1-r)\omega(rz/|z|)+(1-r^2)^2},\quad z\in\DD \setminus (0).
$$
Our goal in this section is to extend this result to the setting of  pure superharmonic weights. In what follows, for a positive Borel measure $\mu$ on $\DD$, we write 
$$
\Psi _\mu (z)= \displaystyle \int _\DD\frac{1-|z|^2}{|1-z\bar{w}|^2}d\mu(w),\quad z\in \DD,
$$
and 
$$
V_\mu(z) =\displaystyle \int _\DD\frac{(1-|z|^2)(1-|w|^2)}{|1-z\bar{w}|^2}d\mu(w),\quad z\in \DD.
$$
We  prove the following result
\begin{theo}\label{kernel}
	Let $\omega$ be a positive  pure superharmonic  weight on $\DD$, and let $\mu$ be the associated positive Borel measure.
	Then 
	$$
	K^\omega(z, z) \asymp 1+ \displaystyle \int _0^{|z|}\frac{dr}{(1-r)V_\mu(rz/|z|)+(1-r)^2},\quad z\in\DD \setminus (0),
	$$
	where the implied constants are universal.	
\end{theo}
The proof of this theorem follows the strategy developed in \cite{EEK18}. We now introduce distance functions that will play a crucial role in the sequel. Let $\varphi : (0,\pi] \to (0,\infty)$ be a continuous decreasing function, and let $E$ be a closed subset of $\TT$ such that
$$
\log \varphi( \dist(.,E)) \in \mathrm{L}^1(\TT).
$$
Let $\tilde {\varphi}_E$ be the outer function associated with $\varphi( \dist(.,E))$ defined by
$$
\tilde {\varphi}_E(z)= \exp \displaystyle \int _\TT\frac {\zeta+z}{\zeta-z}\log \varphi( \dist(\zeta, E))dm(\zeta),\quad z\in \DD.
$$
In the case $E= \{\zeta\}$, we write $\tilde{\varphi}_{\zeta} = \tilde{\varphi}_{\{\zeta\}}$  and we have 
$$|\tilde{\varphi}^*_{1}(e^{it})| = \varphi(|t|) \quad \text{a.e. on } \TT.$$
In what follows, we assume that $\varphi : (0,\pi] \to (0,\infty)$ is a decreasing $C^1$-function. We say that $\varphi \in \mathcal{R}$ if, in addition, the function $x \mapsto x^2 \vert\varphi'(x)\vert$ is increasing and
\begin{equation}\label{funct-doub}
	\varphi(2x) \asymp \varphi(x), \quad \vert\varphi'(2x)\vert \asymp \vert\varphi'(x)\vert.
\end{equation}
The main ingredient in the proof of Theorem~\ref{kernel} is the following estimate.

\begin{theo}\label{NE} 
	Let $\zeta= e^{is_0}\in \TT$. Suppose that $ \log \varphi( \dist(.,E)) \in \mathrm{L}^1(\TT)$  and  that $\varphi \in \mathcal{R}$. Then
	$$
	\cD_\omega(\tilde {\varphi}_{\zeta}) \lesssim \|\varphi' F_{\mu,\zeta} \|_\infty \| \varphi \|_\infty,
	$$
	where 
	$$F_{\mu,\zeta}(y)=\displaystyle \int _{re^{is}\in \DD} \frac{(1-r)y^2}{(1-r)^2+(s-s_0)^2+y^2}d\mu(r,s).$$
\end{theo}
Without loss of generality, we assume that $\zeta=1$. By Theorem \ref{Function}  with $a= 2\log\varphi((1-r)\vee s)$, we obtain 
$$\cD _{re^{is}}(\tilde{\varphi}_1)\lesssim \displaystyle \int_0 ^{\pi} \frac{\varphi ^2(t)-\varphi^2((1-r)\vee s)-2\varphi^2((1-r)\vee s)\log\left ( \frac{\varphi (t)}{\varphi((1-r)\vee s)}\right )}{|e^{it}-re^{is}|^2}dt.$$

\noindent The following elementary inequality will be useful in the proof of Theorem \ref{NE}.
\begin{lem}\label{EI}
	Let $\varphi$ satisfy the assumptions of Theorem~\ref{NE}. Then
	$$
	\displaystyle \int _{x=0}^{\pi}  \displaystyle \int _{y=0}^x\frac{|\varphi '(x)|}{x \varphi(x)}\varphi (y)dydx \lesssim \|\varphi \|_\infty, \quad x\in  (0,\pi],
	$$
	where the implied constant is universal.
\end{lem}
\begin{proof}
	It suffices to use the identity
	\[
	\varphi(y)=\int_{y}^{\pi} |\varphi'(u)|\,du + \varphi( \pi),
	\]
	together with the fact that $x^2|\varphi'(x)|$ is increasing, which implies
	\[
	x\,|\varphi'(x)| \lesssim \varphi(x), \qquad x \in (0,\pi].
	\]
	Consequently, we obtain 
	\begin{align*}
		\int_{x=0}^{\pi} \int_{y=0}^{x} \frac{|\varphi'(x)|}{x \varphi(x)} \varphi(y)\, dy\, dx
		&\asymp \int_{x=0}^{\pi} \int_{y=0}^{x} \frac{|\varphi'(x)|}{x \varphi(x)} \left( \int_{u=y}^{x} |\varphi'(u)| \, du \right) dy\, dx + \|\varphi\|_\infty \\
		&= \int_{u=0}^{\pi} \int_{x=u}^{\pi} \frac{u}{x} \frac{|\varphi'(x)|}{ \varphi(x)} |\varphi'(u)| \, dx\,  du + \|\varphi\|_\infty\\
		&\lesssim \int_{u=0}^{\pi} \int_{x=u}^{\pi} \frac{u}{x^2} |\varphi'(u)| \, dx\,  du + \|\varphi\|_\infty
		\asymp \|\varphi\|_\infty.
	\end{align*}
	This establishes the desired estimate.
\end{proof}
\begin{proof}[Proof of Theorem \ref{NE}]
	To simplify notation, we set $\tilde{\varphi} = \tilde{\varphi}_1$ and $F_\mu=F_{\mu,1}$.  We have
	\[
	\begin{array}{lll}
		\cD _{re^{is}}(\tilde{\varphi}) & \lesssim  & \displaystyle \int_0 ^{\pi} \frac{\varphi ^2(t)-\varphi^2((1-r)\vee s)-2\varphi^2((1-r)\vee s)\log \left ( \frac{\varphi (t)}{\varphi((1-r)\vee s)}\right )}{|e^{it}-re^{is}|^2}dt\\
		\\
		& = & 4 \displaystyle \int _{t=0}^{ \pi}  \displaystyle \int _{x=(1-r) \vee s}^t   \displaystyle \int _{y=(1-r)\vee s}^x\frac{\varphi'(x)}{\varphi(x)}\varphi'(y)\varphi(y)\frac{1}{|e^{it}-re^{is}|^2}dydxdt\\
		\\
		&=:  &  4 I(r,s).
	\end{array}
	\]
	We next decompose $I(r,s)$ into three parts
	$$
	I(r,s)= \displaystyle \int _{t=0}^{s/2}+  \displaystyle \int _{t=s/2}^{2s}+  \displaystyle \int _{t\geq 2s}=: I_1(r,s)+I_2(r,s)+I_3(r,s).
	$$
	For each $j=1,2,3$, we define 
	$$
	I_j(\mu)= \displaystyle \int _\DD (1-r^2)I_j(r,s)d\mu(r,s).
	$$
	Hence
	$$
	\cD_{\mu}(\tilde{\varphi})\,\lesssim \, I_1(\mu)+I_2(\mu)+I_3(\mu).
	$$
	{\bf Estimation of $I_1(\mu)$.}\\
	
	\noindent Note that for $t\in (0,s/2)$, we have ${|e^{it}-re^{is}|^2}\asymp (1-r)^2+ s^2$.
	Then 
	\[
	\begin{array}{lll}
		(1-r^2)I_1(r,s) & \asymp & \displaystyle \int _{t=0}^{s/2}  \displaystyle \int _{x=(1-r) \vee s}^t   \displaystyle \int _{y=(1-r)\vee s}^x\frac{\varphi'(x)}{\varphi(x)}\varphi'(y)\varphi(y)\frac{1-r}{(1-r)^2+s^2}dydxdt\\
		\\
		& \lesssim &  \displaystyle \int _{x=0}^{(1-r)\vee s}  \displaystyle \int _{y=x}^{(1-r) \vee s}  \frac{\varphi'(x)}{\varphi(x)}\varphi'(y)\varphi(y)\left (\displaystyle \int _{t=0}^{x\wedge s}\frac{1-r}{(1-r)^2+s^2}dt\right )dydx\\
		
		& \leq &  \displaystyle \int _{x=0}^{(1-r)\vee s}  \displaystyle \int _{y=x}^{(1-r) \vee s}  \frac{\varphi'(x)}{\varphi(x)}\varphi'(y)\varphi(y)\frac{x(1-r) }{(1-r)^2+s^2}dydx.\\
	\end{array}
	\]
	Integrating with respect to $\mu$, we obtain 
	\[
	\begin{array}{lll}
		I_1(\mu) &\lesssim  &\displaystyle \int _{x=0}^{{(1-r)\vee s}}  \displaystyle \int _{y=x}^{(1-r) \vee s}  \frac{\varphi'(x)}{\varphi(x)}\varphi'(y)\varphi(y)\displaystyle \int _{re^{is}\in \DD}\frac{x(1-r) }{(1-r)^2+s^2}d\mu(r,s)dydx\\
		\\
		& \lesssim &  \displaystyle \int _{x=0}^{(1-r)\vee s}  \displaystyle \int _{y=x}^{(1-r) \vee s}  \frac{\varphi'(x)}{\varphi(x)}\varphi'(y)\varphi(y)\frac{x}{y^2}F_\mu(y)dydx\\
		\\
		& \lesssim & \| \varphi ' F_\mu\|_\infty  \displaystyle \int _{x=0}^{(1-r)\vee s}  \displaystyle \int _{y=x}^{(1-r) \vee s}  \frac{\vert\varphi'(x)\vert}{\varphi(x)}\varphi(y)\frac{x}{y^2}dydx\\
		\\
		& \leq &\| \varphi ' F_\mu\|_\infty \|\varphi\|_\infty.
	\end{array}
	\]
	The last inequality holds since $\varphi$ is decreasing. 
	\\
	
	\noindent {\bf Estimation of $I_2(\mu)$.}\\
	
	First, suppose that $1-r\leq s$. By \eqref{funct-doub}, we get
	\[
	\begin{array}{lll}
		(1-r^2)I_2(r,s) & \asymp & \displaystyle \int _{t=s/2}^{2s}  \displaystyle \int _{x =s}^t   \displaystyle \int _{y= s}^x\frac{\varphi'(x)}{\varphi(x)}\varphi'(y)\varphi(y)\frac{1-r}{(1-r)^2+(t-s)^2}dydxdt\\
		\\
		& \asymp &  \displaystyle \int _{t=s/2}^{2s}  \displaystyle \int _{x =s}^t   \displaystyle \int _{y= s}^x\varphi'(t)^2\frac{1-r}{(1-r)^2+(t-s)^2}dydxdt\\
		\\
		& \asymp & \displaystyle \int _{t=s/2}^{2s} \varphi'(t)^2  \frac{(t-s)^2(1-r)}{(1-r)^2+(t-s)^2}dt\\
		\\
		& \lesssim & \displaystyle \int _{t=s/2}^{2s} \varphi'(t)^2  \frac{s^2(1-r)}{(1-r)^2+s^2+t^2}dt.\\
	\end{array}
	\]
	Consequently
		\begin{equation*}
			I_{21}(\mu):=\displaystyle \int _{\{1-r\leq s\}}(1-r^2)I_2(r,s)d\mu(r,s)  \lesssim    \displaystyle \int _{0}^{\pi} \varphi'(t)^2F_\mu(t)dt \leq \| \varphi ' F_\mu\|_\infty \|\varphi\|_\infty.\\\\
		\end{equation*}
		
		\noindent Next, suppose that  $1-r> s$. Then 
		\[
		\begin{array}{lll}
			(1-r^2)I_2(r,s) & \asymp & \displaystyle \int _{t=s/2}^{2s}  \displaystyle \int _{x =t}^{1-r}   \displaystyle \int _{y= x}^{1-r}\frac{\varphi'(x)}{\varphi(x)}\varphi'(y)\varphi(y)\frac{1-r}{(1-r)^2+s^2+y^2}dydxdt\\
			\\
			& \lesssim & \displaystyle \int _{x =s/2}^{1-r}   \displaystyle \int _{y= x}^{1-r}\frac{\varphi'(x)}{\varphi(x)}\varphi'(y)\varphi(y)\frac{s(1-r)}{(1-r)^2+s^2+y^2}dydx.\\
		\end{array}
		\]
		Using this estimate, together with the fact that $\varphi$ is decreasing, we deduce 
		\begin{align*}
			I_{22}(\mu)
			&:= \int_{\{1-r> s\}} (1-r^2)\, I_2(r,s)\, d\mu(r,s)
			\\
			&\lesssim \int_{x=0}^{ \pi} \int_{y=x}^{ \pi}
			\frac{\varphi'(x)}{\varphi(x)}\, \varphi'(y)\varphi(y)
			\left(\int_{re^{is}\in \DD} \frac{x(1-r)}{(1-r)^2+s^2+y^2}\, d\mu(r,s)\right)
			dy\,dx
			\\ 
			&= \int_{x=0}^{ \pi} \int_{y=x}^{ \pi}
			\frac{\varphi'(x)}{\varphi(x)}\, \varphi'(y)\varphi(y)\,
			\frac{x}{y^2}\, F_\mu(y)\, dy\,dx
			\\ 
			&\lesssim \|\varphi' F_\mu\|_\infty
			\int_{x=0}^{ \pi} \int_{y=x}^{ \pi}
			\vert\varphi'(x)\vert\,\frac{x}{y^2}\, dy\,dx
			\\ 
			&\lesssim \|\varphi' F_\mu\|_\infty \,\|\varphi\|_\infty.
		\end{align*}
		Combining the above estimates, we conclude that 
		$$
		I_2(\mu) =I_{21}(\mu)+I_{22}(\mu)  \lesssim  \|\varphi'F_\mu\|_\infty \|\varphi\|_\infty.
		$$
		
		\noindent {\bf Estimation of $I_3(\mu)$.}\\
		
		Two cases are distinguished in the estimation of $I_3(\mu)$. In the first case, we assume that $1 - r \leq s$. Thus, we obtain  
		\[
		\begin{array}{lll}
			(1-r^2)I_3(r,s) & \lesssim & \displaystyle \int _{t=s}^{ \pi} \displaystyle \int _{x =s}^t   \displaystyle \int _{y=s}^x\frac{\varphi'(x)}{\varphi(x)}\varphi'(y)\varphi(y)\frac{1-r}{(1-r)^2+t^2}dydxdt\\
			\\
			& = & \displaystyle \int _{x =s}^{ \pi}   \displaystyle \int _{y=s}^x\frac{\varphi'(x)}{\varphi(x)}\varphi'(y)\varphi(y) \left (\displaystyle \int _{t=x}^{\pi} \frac{1-r}{(1-r)^2+t^2}dt\right )dydx\\
			\\
			&\asymp & \displaystyle \int _{x =s}^{ \pi}   \displaystyle \int _{y=s}^x\frac{\varphi'(x)}{x\varphi(x)}\varphi'(y)\varphi(y)(1-r)dydx.
		\end{array}
		\]
		Integrating over the region $\{1-r\leq s\}$ yields
		\begin{small}{
				\[
				\begin{aligned}
					\int_{\{1-r\le s\}} (1-r^2) I_3(r,s)\, d\mu(r,s)
					&\lesssim
					\int_{1-r\le s}
					\int_{x=s}^{\pi}\int_{y=s}^{x}
					\frac{\varphi'(x)}{x\varphi(x)}\varphi'(y)\varphi(y)(1-r)
					\,dy\,dx\,d\mu(r,s) \\
					&\lesssim
					\int_{x=0}^{\pi}\int_{y=0}^{x}
					\frac{\varphi'(x)}{x\varphi(x)}\varphi'(y)\varphi(y)\int_{1-r\le s}
					(1-r)\,d\mu(r,s) 
					\,dy\,dx\\
					&\lesssim
					\int_{x=0}^{\pi}\int_{y=0}^{x}
					\frac{\varphi'(x)}{x\varphi(x)}\varphi'(y)\varphi(y)
					F_\mu(y)\,dy\,dx \\
					&\lesssim
					\|\varphi' F_\mu\|_\infty
					\int_{x=0}^{\pi}\int_{y=0}^{x}
					\frac{|\varphi'(x)|}{x\varphi(x)}\varphi(y)\,dy\,dx \\
					&\lesssim
					\|\varphi' F_\mu\|_\infty \|\varphi\|_\infty .
				\end{aligned}
				\]}\end{small}
		The last inequality follows from Lemma \ref{EI}.
		
		In the second case, suppose that $1-r>s$. We decompose
		\[
		\begin{array}{lll}
			(1-r^2)I_3(r,s)& \lesssim & \displaystyle \int _{t=s}^{1-r}  \displaystyle \int _{x=t}^{1-r}   \displaystyle \int _{y=x}^{1-r}\frac{\varphi'(x)}{\varphi(x)}\varphi'(y)\varphi(y)\frac{1-r}{(1-r)^2+t^2}dydxdt\\
			& &+ \displaystyle \int _{t=1-r}^{\pi}  \displaystyle \int _{x=1-r}^{t}   \displaystyle \int ^{x}_{y=1-r}\frac{\varphi'(x)}{\varphi(x)}\varphi'(y)\varphi(y)\frac{1-r}{(1-r)^2+t^2}dydxdt\\
			\\
			& {=:}&  J_1(r,s)+J_2(r,s).
		\end{array}
		\]
		First, we observe that 
		\[
		\begin{array}{lll}
			J_1(r,s) & \asymp & \displaystyle \int _{t=s}^{1-r}  \displaystyle \int _{x=t}^{1-r}   \displaystyle \int _{y=x}^{1-r}\frac{\varphi'(x)}{\varphi(x)}\varphi'(y)\varphi(y)\frac{1}{1-r}dydxdt\\
			\\
			& \lesssim & \displaystyle \int _{x=s}^{1-r}  \displaystyle \int _{y=x}^{1-r}   \frac{\varphi'(x)}{\varphi(x)}\varphi'(y)\varphi(y)\frac{x(1-r)}{(1-r)^2+ s^2+y^2}dydx.\\
		\end{array}
		\] 
		It follows from this estimate and the monotonicity of $\varphi$ that
		\[
		\begin{array}{lll}
			\displaystyle \int_{\{ s< 1-r \}}J_1(r,s) d\mu(r,s)& \lesssim & \displaystyle \int _{x=0}^{\pi}  \displaystyle \int _{y=x}^{\pi}   \frac{\varphi'(x)}{\varphi(x)}\varphi'(y)\varphi(y)\displaystyle \int_{\{ s< 1-r \}} \frac{x(1-r)}{(1-r)^2+ s^2+y^2}d\mu(r,s)dydx\\
			\\
			& \lesssim & \displaystyle \int _{x=0}^{\pi}  \displaystyle \int _{y=x}^{\pi}   \frac{\varphi'(x)}{\varphi(x)}\varphi'(y)\varphi(y)\frac{x}{y^2}F_\mu(y)dydx\\
			\\
			& \lesssim & \|\varphi ' F_\mu\|_\infty \displaystyle \int _{x=0}^{\pi}  \displaystyle \int _{y=x}^{\pi}   |\varphi'(x)|\frac{x}{y^2}dydx\\
			\\
			& \lesssim & \|\varphi ' F_\mu\|_\infty \| \varphi \|_\infty.
		\end{array}
		\]
		On the other hand,
		\[
		\begin{array}{lll}
			J_2(r,s) & \asymp & \displaystyle \int _{t=1-r}^{\pi}  \displaystyle \int _{x=1-r}^{t}   \displaystyle \int _{y=1-r}^{x}\frac{\varphi'(x)}{\varphi(x)}\varphi'(y)\varphi(y)\frac{1-r}{t^2}dydxdt\\
			\\
			& \lesssim & \displaystyle \int _{x=1-r}^\pi  \displaystyle \int _{y=1-r}^x   \frac{\varphi'(x)}{\varphi(x)}\varphi'(y)\varphi(y)\frac{1-r}{x}dydx.\\
		\end{array}
		\]
		Therefore by Lemma \ref{EI}, we get
		\[
		\begin{array}{lll}
			\displaystyle \int_{\{s< 1-r\}}J_2(r,s) d\mu(r,s)& \lesssim &  \displaystyle \int _{x=0}^\pi  \displaystyle \int _{y=0}^x   \frac{\varphi'(x)}{\varphi(x)}\varphi'(y)\varphi(y)\displaystyle \int _{ {\{s< 1-r<y\}}}\frac{1-r}{x}d\mu(r,s)dydx\\
			\\
			& \lesssim &  \displaystyle\int_{x=0}^\pi  \int _{y=0}^x   \frac{\varphi'(x)}{\varphi(x)}\varphi'(y)\varphi(y) {\frac{F_\mu(y)}{x}}dydx\\
			\\
			& \lesssim & \|\varphi ' F_\mu\|_\infty \displaystyle \int _{x=0}^\pi  \displaystyle \int _{y=0}^x\frac{{\vert\varphi'(x)\vert}}{x\varphi(x)}\varphi(y)dydx\\
			& \lesssim &  \|\varphi ' F_\mu\|_\infty \| \varphi \|_\infty.
		\end{array}
		\]
		\[
		\begin{array}{lll}
			\displaystyle \int_{\{s<1-r\}} J_2(r,s)\, d\mu(r,s)
			& \lesssim &
			\displaystyle \int_{x=0}^{\pi} \int_{y=0}^{x}
			\frac{\varphi'(x)}{\varphi(x)} \varphi'(y)\varphi(y)
			\int_{{\{s<1-r<y\}}} \frac{1-r}{x}\, d\mu(r,s) dy\, dx \\[1.2em]
			& \lesssim &
			\displaystyle \int_{x=0}^{\pi} \int_{y=0}^{x}
			\frac{\varphi'(x)}{\varphi(x)} \varphi'(y)\varphi(y)\,
			{\frac{F_\mu(y)}{x}}\, dy\, dx \\[1.2em]
			& \lesssim &
			\|\varphi' F_\mu\|_\infty
			\displaystyle \int_{x=0}^{\pi} \int_{y=0}^{x}
			\frac{{|\varphi'(x)|}}{x\varphi(x)} \varphi(y)\, dy\, dx \\[1.2em]
			& \lesssim &
			\|\varphi' F_\mu\|_\infty \, \|\varphi\|_\infty.
		\end{array}
		\]
		The desired estimate for $I_3(\mu)$ follows by combining the above estimates. This completes the proof.
	\end{proof}
	\begin{proof}[Proof of Theorem \ref{kernel}]
		Let $\omega$ be a positive pure superharmonic weight, and let $\mu$ denote the associated positive Borel measure. Recall that we assume $\zeta =1$ and set $F_\mu:= F_{\mu,1}$. We observe that
		\[
		F_\mu(y) \asymp y\, V_\mu(1 - y), \qquad y \in (0,1).
		\]
		Therefore we aim to prove the estimate
		$$
		K^\omega(r, r) \asymp 1+ \displaystyle \int ^1_{1-r}\frac{dy}{F_\mu(y)+y^2},\quad r\in (0,1).
		$$
		By the Paley-Littlewood formula,
		$$
		\| f\| ^2_{\cD _\omega} \asymp |f(0)|^2+ \displaystyle \int _\DD |f'(z)|^2(V_\mu(z)+\log(1/|z|^2)dA(z).
		$$
		The upper bound can be obtained by following the proof of \cite[Theorem~1]{EEK18}.\\
		
		\noindent 
		It remains to establish the lower bound. To this end, we first observe that, by the Cauchy--Schwarz inequality, we have
		\begin{equation}\label{CS}
			|f(r)|^2 \leq \| f \|^2_{\mathcal{D}_\omega}\, K^{\omega}(r,r), \quad f \in \mathcal{D}_\omega.
		\end{equation}
		It is clear that $1 \leq K^{\omega}(r,r)$. On the other hand, let $\varphi$ be the function defined by
		\begin{equation}\label{FI}
			\varphi(t) = \int_{t+(1-r)}^\pi \frac{dy}{F_\mu(y) + y^2}.
		\end{equation}
		Since $F_\mu$ is increasing and the function $y \mapsto F_\mu(y)/y^2$ is decreasing, it follows that $\varphi \in \mathcal{R}$ and $\|\varphi' F_\mu\|_{\infty} \leq 1$. Hence by \eqref{CS} and Theorem~\ref{NE}, we obtain
		\[
		|\tilde{\varphi}_1(r)|^2 \leq \| \tilde{\varphi}_1 \|^2_{\mathcal{D}_\omega} K^{\omega}(r,r)
		\lesssim \varphi(0)\, K^{\omega}(r,r).
		\]
		Moreover, Proposition 2.3 in \cite{EEK18} yields
		\[
		\varphi(1-r) \lesssim \tilde{\varphi}_1(r).
		\]
		Since $\varphi(0) \asymp \varphi(1-r)$, we conclude that
		\[
		\varphi(0) \lesssim K^{\omega}(r,r).
		\]
		This implies
		\[
		\int_{1-r}^\pi \frac{dy}{F_\mu(y) + y^2} \lesssim K^{\omega}(r,r),
		\]
		which completes the proof.
	\end{proof}
	
	
	\section{Capacity associated with $\cD_\omega$}\label{capomega}
	
	
	
	Throughout this paper, we assume that $\omega$ is a positive superharmonic function on $\DD$.
	Since $\mathcal{D}_\omega \subset \mathrm{H}^2$, every function $f \in \mathcal{D}_\omega$
	admits boundary values $f^*$ on $\mathbb{T}$, defined almost everywhere with respect to $m$.
	Moreover, $f^*$ exists $\nu$-almost everywhere for $f \in \mathcal{D}_\omega$. A Douglas-type formula for $\mathcal{D}_\omega$
	can be expressed entirely in terms of $f^*$; see the following theorem.
	\begin{theo}{(\cite{BGP18,RS1})} \label{douglas}
		Let \(\omega\) be a positive superharmonic function on \(\mathbb{D}\), with \(\mu\) and \(\nu\) as the Riesz and boundary measures of \(\omega\), respectively. If \(f \in \cD_\omega\), then
		\[
		\cD_\omega(f) = \int_{\mathbb{T}} \int_{\mathbb{T}} |f^*(\zeta) - f^*(\lambda)|^2 \left( A_\mu(\zeta, \lambda) \, dm(\zeta) + \frac{d\nu(\zeta)}{|\zeta - \lambda|^2} \right) dm(\lambda),
		\]
		where 
		\[
		A_\mu(\zeta, \lambda) = \int_{\mathbb{D}} \frac{1 - |z|^2}{|\zeta - z|^2} \cdot \frac{1 - |z|^2}{|\lambda - z|^2} \, d\mu(z).
		\]
	\end{theo}	
	Our aim is to obtain the definition of the capacity associated with $\cD_\omega$. In order to do this, we introduce the harmonic version of $\cD_\omega$, denoted by $\cD_\omega^h$
	\begin{small}$$
		\cD_\omega^h:=\left\{f\in\IL:\cD_\omega^h(f):=\int_{\TT}\int_{\TT}|f(\zeta)-f(\lambda)|^2\left(A_\mu(\zeta,\lambda)dm(\zeta)+\frac{d\nu(\zeta)}{|\zeta-\lambda|^2}\right)dm(\lambda)<\infty\right\},
		$$\end{small}
	where $\mu$ and $\nu$ are respectively the Riesz measure and boundary measure of $\omega$. The space $\cD_\omega^h$ is equipped with the inner product 
	$$
	\left<f,g\right>_{\cD_\omega^h}:=\left<f,g\right>_{\IL}+\cD_\omega^h(f,g), \quad f,g\in\cD_\omega^h,
	$$
	where 
	$$
	\cD_\omega^h(f,g):=\int_{\TT}\int_{\TT}\left(f(\zeta)-f(\lambda)\right)\overline{\left(g(\zeta)-g(\lambda)\right)}\left(A_\mu(\zeta,\lambda)dm(\zeta)+\frac{d\nu(\zeta)}{|\zeta-\lambda|^2}\right)dm(\lambda).
	$$
	The space $\cD_\omega^h$ is a reproducing kernel Hilbert space, and its reproducing kernel is given by
	\(
	k^{\omega,h} = 2\Re\, K^\omega - 1.
	\) Moreover, from Theorem \ref{douglas}, $\cD_\omega^h$ contains $\cD_\omega$ as a closed subspace. 
	
	Furthermore the semi--inner product $\cD_\omega^h(.,.)$ is an interesting example of a Dirichlet form, and $\cD_\omega^h$ is a Dirichlet space in the sense of Beurling--Deny (see \cite{BD, Fuk}). The capacity $c_\omega$ associated with $\cD_\omega^h$  of a subset $E\subset \TT$ is then defined by 
	$$
	c_\omega(E)=\inf\{||f||^2_{\cD_\omega^h}:\,\,f\in \cD_\omega^h\,\,\text{and}\,\,|f|\geq 1\,\,m-a.e.\,\,\text{on a neighbourhood of}\,\, E\}.
	$$
	Since $\cD_\omega^h(.,.)$ is Dirichlet form, the capacity $c_\omega$ is a Choquet capacity, and $c_\omega$ satisfies the capacity weak--type inequality; namely,  
	$$
	c_\omega\left(|f(\zeta)|>t\right)\leq \frac{||f||^2_{\cD_\omega^h}}{t^2}.
	$$
	\begin{prop}\label{arcestimate}
		Let $I$ be an arc of $\TT$, and let $z_{I}=(1-|I|)\zeta$, with $\zeta$ is the midpoint of $I$. Then 
		\begin{equation*}\label{estimatecap}
			\frac{1}{K^\mu(z_I,z_I)}\asymp c_{\mu}(I).
		\end{equation*}
	\end{prop}
	\begin{proof} 
		Indeed, first recall that $$c_{\mu}(I)=\inf\left\lbrace \|f\|^2_{\cD^h_\mu}: f\in \cD^h_\mu \cap C^1, \,\, 0\leq f\leq 1, \,\,f=1\,\,\, \text{on}\,\,\, I \right\rbrace.
		$$
		Let $f \in \mathcal{D}^h_\mu \cap C^1$ be such that $0 \leq f \leq 1$ and $f=1$ on $I$. Then, for $|I|$ sufficiently small, the Poisson integral of $f$ satisfies
		\[
		P[f](z_I)\asymp 1.
		\]Thus, 
		\begin{eqnarray*}
			1&\lesssim& |P[f](z_I)|= |f^+(z_I)+f^-(z_I)| \leq |f^+(z_I)|+|f^-(z_I)| \\
			\\
			&\leq& \|f^+\|_{\cD_\mu}.\|K^{\mu}_{z_I}\|_{\cD_\mu}+\|f^-\|_{\cD_\mu}.\|K^{\mu}_{z_I}\|_{\cD_\mu}\\
			\\
			&\leq& 2\|f\|_{\cD_\mu}.\|K^{\mu}_{z_I}\|_{\cD_\mu}.
		\end{eqnarray*}
		This implies that 
		$$\frac{1}{K^\mu(z_I,z_I)}\lesssim c_{\mu}(I).$$
		The reverse inequality from Theorem \ref{kernel} and the following inequality
		$$c_\mu(I) \lesssim \| \tilde{\varphi}_\zeta/\tilde{\varphi}(1-|I|)\|,$$
		where $\varphi$ is given by (\ref{FI}).
	\end{proof}
	\begin{coro}
		Let $\mu$ be a positive Borel measure on \(\DD\) satisfying (\ref{RMF}), and let $\zeta \in \TT$. Then 
		$c_\mu (\{\zeta\})= 0$ if and only if 
		$$
		\displaystyle \int _0^{1}\frac{dr}{(1-r)V_\mu(rz/|z|)+(1-r)^2}= \infty.
		$$
	\end{coro}

	The capacity $c_\omega$ also satisfies the strong--type capacitary inequality
	
	\begin{theo}\label{Strong-inequality}
		Let $f\in \cD^h_\omega$, then 
		$$\int_0^{\infty}c_{\omega}(|f|>t)tdt\lesssim \|f\|^2_{\cD_\omega^h}.$$
	\end{theo}
	\begin{proof}
		A proof in a general framework can be found in Lemma 2.4.1, p.~101 in \cite{Fuk} 
	\end{proof}
	
	Using the same techniques as in \cite{EL}, the capacity $c_\omega$ satisfies the following converse of the strong--type inequality.
	%
	\begin{theo}\label{ConverseStrongInequality}
		Let $\Ee$ be a closed subset of $\TT$, and let  $\eta : (0,\pi]\rightarrow (0,+\infty)$ be a decreasing, continuous function such that $\eta(0^+)=\infty$. The following statements are equivalent
		\begin{enumerate}
			\item[A.] There exists a function $f\in \cD_\omega$ such that 
			\begin{equation*}\label{BB1}
				\liminf_{z\rightarrow \zeta}\vert f(z)\vert\geq \eta(d(\zeta,\Ee)), \qquad \zeta \in \TT. 
			\end{equation*}
			\item[B.] There exists a function $f\in \cD_\omega$ such that 
			\begin{equation*}\label{RR1}
				\liminf_{z\rightarrow \zeta}\Re f(z)\geq \eta(d(\zeta,\Ee)), \qquad \zeta \in \TT.
			\end{equation*}
			\item[C.] The function $\eta$ and the set $\Ee$ satisfy 
			\begin{equation*}\label{CC1}
				\int_0^\pi c_\omega (\Ee_t)\vert d\eta^2(t)\vert<\infty.
			\end{equation*} 
		\end{enumerate}
	\end{theo}

	
	\section{Structure of invariant subspaces of $\cD_\omega$}\label{Structure}
	
	
	In this section, we study the structure of shift--invariant subspaces of Dirichlet spaces associated with superharmonic weights. Throughout, we assume that $\omega$ is a positive superharmonic weight determined by measures $\mu$ on $\DD$ and $\nu$ on $\TT$.
	\subsection{Invariant subspaces and extremal functions}
	Extremal functions in $\mathcal{D}_\omega$ play a central role in the study of invariant subspaces of $\mathcal{D}_\omega$. We recall basic facts about these functions and state results due to Richter--Sundberg and Aleman.
	
	We say that a function $\phi \in \cD_\omega$ is an \textit{extremal function} in \(\cD_\omega\) if 
	$$
	\langle \phi , z^n\phi \rangle_{\cD_\omega} =\delta_{n,0} \quad {\text{and}} \quad \phi^{(p)}(0)>0,\quad n\ge 1,
	$$
	where $p$ is the smallest integer such that $\phi^{(p)}(0) \neq 0$. Note that, if we write $\cM _\phi= [\phi]_{\cD_\omega}$,  then $\phi$ is the unique solution to the extremal problem
	$$
	\sup \{\Re f^{(p)}(0)\,\,: \,\,f \in \cM_\phi\}.
	$$
	One has
	\[
	\|\phi f\|_{\cD_\omega} = \|f\|_{\cD_{\omega_\phi}}, \quad f \in \mathcal{D}_\omega,
	\]
	where $\omega_\phi$ is the superharmonic weight associated with the measures $d\mu_\phi = |\phi|^2\, d\mu$ on $\mathbb{D}$ and $d\nu_\phi = |\phi^*|^2\, d\nu$ on $\mathbb{T}$.
	
	Aleman \cite{Al} showed that if $\varphi$ is an extremal function of $\cD_\omega$, then $\varphi$ is a multiplier of $\cD_\omega$ and satisfies $|\varphi(z)| \le 1$ for all $z \in \DD$.
	
	In the special case where $\omega = P_\nu$ is harmonic, with $\nu$ a finite positive Borel measure on $\mathbb{T}$, Richter~\cite{Ri1} showed that $S_\nu$ is a two--isometry, that is,
	\[
	S_\nu^{*2} S_\nu^2 - 2 S_\nu^* S_\nu + I = 0.
	\]
	He further proved that $(S_\nu, \mathcal{D}_\nu)$ provides a functional model for cyclic analytic two--isometries. Later, Richter and Sundberg~\cite{RS2} characterized the lattice $\operatorname{Lat}(S_\nu)$ of invariant subspaces.
	
	\begin{theo}[Richter--Sundberg]\label{RSdescription}
		Let $\nu$ be a positive finite Borel measure on $\mathbb{T}$, and let $\mathcal{M} \in \operatorname{Lat}(S_\nu) \setminus \{0\}$. Then there exists a unique extremal function $\phi \in \mathcal{D}_\nu$ such that
		\begin{equation}\label{extremal}
			\mathcal{M} = [\phi]_{\mathcal{D}_\nu} = \phi \mathcal{D}_{\nu_\phi}.
		\end{equation}
		Moreover, if $\phi$ admits the factorization $\phi = \mathcal{I} f$, where $\mathcal{I}$ is inner and $f$ is outer, then
		\begin{equation}\label{generated}
			\mathcal{M} = \mathcal{I} \hH \cap [f]_{\mathcal{D}_\nu}.
		\end{equation}
	\end{theo}
	Aleman~\cite{Al} showed that the operator $S_\omega$ associated with a superharmonic weight $\omega$ satisfies
	\begin{equation}\label{supernormal}
		\sum_{k=1}^n (-1)^k \binom{n}{k} S_\omega^{*k} S_\omega^k \le 0,
		\qquad n \ge 2.
	\end{equation}
	Moreover, $S_\omega$ is a two--isometry if and only if $\omega$ is harmonic. Together with the fact that $S_\omega$ is cyclic and analytic, this condition leads to an extension of~\eqref{extremal} to the class of superharmonic weights $\omega$.
	\begin{theo}[Aleman]\label{Aleman2}
		Let $\omega$ be a positif superharmonic weight on $\DD$. Let $\cM\in \Lat(S_\omega)\setminus \{0\}$. Then there exists a unique extremal function $\phi\in\cD_\omega$ such that 
		$$
		\cM=\phi\cD_{\omega_\phi}.
		$$
	\end{theo}
	Our aim is to extend the representation~\eqref{generated} in Theorem~\ref{RSdescription} to superharmonic weights. This reduction allows us to restrict attention to invariant subspaces generated by outer functions. For further details, see Subsection~\ref{inner--outer factorization}.

	The following result will be very useful in several arguments. For the proof see \cite{Al}.  
	\begin{prop}\label{coro1}
		Let $\omega$ be a positif superharmonic weight on $\DD$.
		\begin{enumerate}
			\item[A.] Let $\mathcal{M}$ be a closed subspace of $\cD_\omega$. Suppose there exists a sequence $(g_n)$ in $\mathcal{M}$ such that $\sup_{n} \cD_\omega(g_n) < \infty$ and $g_n \to g$ pointwise in $\mathbb{D}$. Then $g \in \mathcal{M}$.\\
			\item[B.] Let $f,g\in \cD_\omega$ such that $\vert g \vert \leq \vert f \vert $ on $ \D$. Then $\left[  g  \right]_{\cD_\omega}\subset \left[  f  \right]_{\cD_\omega}$. \\
		\end{enumerate}
	\end{prop}
	\subsection{Invariant subspaces and cut--off functions}
	
	Let $f$ be an outer function in $\mathcal{D}_\omega$, and let $\mathcal{I}$ be an inner function such that $\mathcal{I}f, \mathcal{I}f^2 \in \mathcal{D}_\omega$. Set $\varphi_n := \mathcal{I}(f \wedge n f^2)$. Then $(\varphi_n)$ is bounded in $[\mathcal{I}f^2]_{\mathcal{D}_\omega}$ and converges pointwise in $\mathbb{D}$ to $\mathcal{I}f$. Hence by Proposition~\ref{coro1},
	\[
	[\mathcal{I}f]_{\mathcal{D}_\omega} \subset [\mathcal{I}f^2]_{\mathcal{D}_\omega}.
	\]
	
	Since the invariant subspaces generated by $f$ and $f \wedge 1$ coincide (see~\cite{Al}), we may assume that $f \in \mathcal{D}_\omega \cap \kH$. Let $\alpha>0$ with $f^\alpha \in \mathcal{D}_\omega$. Then, by Proposition~\ref{coro1} and the preceding inclusion,
	\[
	[f]_{\mathcal{D}_\omega} = [f^\alpha]_{\mathcal{D}_\omega}.
	\]
	
	Let $g$ be an outer function in $\mathcal{D}_\omega$. By Theorem~\ref{Aleman2}, and following the same argument as in the harmonic case (see~\cite[Theorem~4.1]{RS2}), we obtain
	\[
	[f \vee g]_{\mathcal{D}_\omega} \subseteq [f,g]_{\mathcal{D}_\omega}.
	\]
	and the reverse inclusion follows from Proposition~\ref{coro1}, so that
	\[
	[f \vee g]_{\mathcal{D}_\omega} = [f,g]_{\mathcal{D}_\omega}.
	\]
	
	Assume that $fg \in \mathcal{D}_\omega$. Then
	\[
	[fg]_{\mathcal{D}_\omega}
	= [(fg)\wedge 1]_{\mathcal{D}_\omega}
	\subseteq [(f(g \wedge 1)) \vee f]_{\mathcal{D}_\omega}
	= [f(g \wedge 1), f]_{\mathcal{D}_\omega}
	= [f]_{\mathcal{D}_\omega},
	\]
	and hence
	\[
	[fg]_{\mathcal{D}_\omega} \subseteq [f]_{\mathcal{D}_\omega} \cap [g]_{\mathcal{D}_\omega}.
	\]
	Since invariant subspaces are generated by multipliers, there exists an outer function $\varphi$ such that
	\[
	[f]_{\mathcal{D}_\omega} \cap [g]_{\mathcal{D}_\omega} = [\varphi]_{\mathcal{D}_\omega}.
	\]
	Arguing as above, $\varphi^2 \in [fg]_{\mathcal{D}_\omega}$, which yields
	\[
	[f]_{\mathcal{D}_\omega} \cap [g]_{\mathcal{D}_\omega} \subseteq [fg]_{\mathcal{D}_\omega}.
	\]
	Thus
	\[
	[f]_{\mathcal{D}_\omega} \cap [g]_{\mathcal{D}_\omega} = [fg]_{\mathcal{D}_\omega}.
	\]
	
	In the case where $fg \notin \mathcal{D}_\omega$, we have that $f$ and $f \wedge 1$ generate the same invariant subspace, and $(f \wedge 1)(g \wedge 1) \in \mathcal{D}_\omega \cap \kH$. Applying the previous identity and Proposition~\ref{coro1}, we obtain
	\[
	[f]_{\mathcal{D}_\omega} \cap [g]_{\mathcal{D}_\omega} = [f \wedge g]_{\mathcal{D}_\omega}.
	\]
	
	We summarize these results in the following theorem.
	\begin{theo}\label{genspace}
		Let $\omega$ be a positif superharmonic weight on $\DD$, and let $f,g \in\cD_\omega$ be outer functions. Let $\mathcal{I}$ be an inner function in $\hH$.Then 
		
		\begin{enumerate}
			\item[A.] $ [\mathcal{I}f]_{\cD_\omega}\subset [\mathcal{I}f^2]_{\cD_\omega}$, if $\mathcal{I}f,\mathcal{I}f^2\in\cD_\omega$,	\\
			\item[B.] $[f^{\alpha}]_{\cD_\omega}=[f]_{\cD_\omega}$, if $f^{\alpha}\in \cD_\omega$ for $\alpha>0$,\\
			\item[C.] $[f,g]_{\cD_\omega}=[f\vee g]_{\cD_\omega}$,\\
			\item[D.] $[f]_{\cD_\omega}\cap [g]_{\cD_\omega}=[fg]_{\cD_\omega}$, if  $fg\in\cD_\omega$, \\
			\item[E.] $[f]_{\cD_\omega}\cap [g]_{\cD_\omega}=[f\wedge g]_{\cD_\omega}$.
			
		\end{enumerate}
	\end{theo}
	\subsection{Invariant subspaces and inner--outer factorization}\label{inner--outer factorization}
	
	In this subsection, we start by giving the proof of Theorem~\ref{RSinvar}, which concerns invariant subspaces and inner--outer factorization.
	\begin{proof}[Proof of Theorem \ref{RSinvar}]
		We prove that if $\mathcal{I}$ is inner and $f$ is outer with $\mathcal{I}f \in \cD_\omega$, then
		\[
		[\mathcal{I}f]_{\cD_\omega} = \mathcal{I} \hH \cap [f]_{\cD_\omega}.
		\]
		One has $[\mathcal{I}f]_{\cD_\omega} \subset [\mathcal{I}f]_{\hH} = \mathcal{I} \hH$, and by the second assertion of Proposition~\ref{coro1},  we obtain
		\[
		[\mathcal{I}f]_{\cD_\omega} \subset [f]_{\cD_\omega}.
		\]
		Thus
		\[
		[\mathcal{I}f]_{\cD_\omega} \subset \mathcal{I} \hH \cap [f]_{\cD_\omega}.
		\]
		For the reverse inclusion, write the closed invariant subspace
		\[
		\mathcal{I} \hH \cap [f]_{\cD_\omega} = \varphi \cD_{\omega_\varphi} = [\varphi]_{\cD_\omega},
		\]
		where $\varphi$ is the associated extremal function in $\cD_\omega$. It suffices to show that $\varphi \in [\mathcal{I}f]_{\cD_\omega}$.
		
		Note that the inner factor of $\varphi$ is $\mathcal{I}$, so that $\varphi = \mathcal{I}\psi$, where $\psi$ is an outer function in $\cD_\omega$. We first show that $\psi \in [f]_{\cD_\omega}$. Let $\chi$ be the extremal function associated with $[f]_{\cD_\omega}$, so that
		\[
		[f]_{\cD_\omega} = \chi \cD_{\omega_\chi}.
		\]
		Since $f$ is outer, so is $\chi$. As $\varphi \in [f]_{\cD_\omega}$, there exists $g \in \cD_{\omega_\chi}$ such that
		\[
		\varphi = \chi \mathcal{I} g.
		\]
		Hence
		\[
		\psi = \chi g \in \chi \cD_{\omega_\chi} = [f]_{\cD_\omega}.
		\]
		Therefore there exists a sequence of polynomials $(p_n)$ such that $p_n f \to \psi$. Since $\varphi = \mathcal{I}\psi$ is a multiplier of $\cD_\omega$, it follows that $(\varphi p_n f)_n \subset [\mathcal{I}f]_{\cD_\omega}$ and converges to $\varphi \psi$. Hence
		\[
		\mathcal{I} \psi^2 = \varphi \psi \in [\mathcal{I}f]_{\cD_\omega}.
		\]
		By Proposition~\ref{coro1}, we have
		\[
		\varphi_n := \mathcal{I}(\psi \wedge n\psi^2) \in [\mathcal{I}f]_{\cD_\omega}.
		\]
		Moreover, $\varphi_n(z) \to \varphi(z)$. From  Theorem~\ref{cut--off}~C.,
		\[
		\cD_\omega(\mathcal{I}(\psi \wedge n\psi^2)) \leq 4\cD_\omega(\phi).
		\]
		By the first assertion of Proposition~\ref{coro1}, we conclude that $\varphi \in [\mathcal{I}f]_{\cD_\omega}$. This completes the proof.
	\end{proof}

	
	
	We say that two Borel  subsets $E_1, \ E_2 \subset \TT$ are equals $c_\omega$-q.e if 
	$$
	c_\omega \left ( (E_1\setminus E_2) \cup (E_2\setminus E_1)\right )= 0.
	$$
	Note that if $h \in \kH$, then $\mathcal{D}_\omega$ is continuously embedded into $\mathcal{D}_{\omega_h}$, where $\omega_h$ is the superharmonic weight associated with the measures $d\mu_h = |h|^2 d\mu$ on $\mathbb{D}$ and $d\nu_h = |h^*|^2 d\nu$ on $\mathbb{T}$. In particular, for every Borel set $E \subset \TT$, one has
	\[
	c_\omega(E) = 0 \quad \Longrightarrow \quad c_{\omega_h}(E) = 0.
	\]
	In the next result, we show that a positive answer to the Brown--Shields conjecture for a class of measures $\mathcal{C}$ satisfying a hereditary property leads to a complete description of invariant subspaces for this class. Here, $\mathcal{C}$ denotes a class of positive Borel measures on $\mathbb{D}$.
	
	We say that $\mathcal{C}$ satisfies property $(\mathrm{H})$ if, whenever $\mu \in \mathcal{C}$ and $\phi$ is an extremal function for $\mathcal{D}_\mu$, we have $|\phi|^2\, d\mu \in \mathcal{C}$.
	
	We then have the following result.
	\begin{theo}\label{equivalence}
		Let $\mathcal{C}$ be a class of measures on $\mathbb{D}$ satisfying property $(\mathrm{H})$ for which the Brown--Shields conjecture holds. Let $\mu \in \mathcal{C}$, and let $\mathcal{M}\in \operatorname{Lat}(S_\mu)\setminus\{0\} $. Then, there exists a Borel subset \(\Ee\) of \(\TT\) such that 
		\[
		\mathcal{M}
		= \mathcal{I}_{\mathcal{M}}\hH
		\cap \{ g \in \mathcal{D}_\mu : g = 0 \ \text{on } E \ \text{$c_\mu$-q.e.} \}.
		\]
	\end{theo}

	\begin{proof}
		Let $\cM$ be a closed invariant subspace of $\cD_\mu$. By Theorem~\ref{RSinvar}, there is an  outer  $f\in   \cD_\omega\cap \kH$ such that 
		\[
		\cM=\mathcal{I}_\cM \hH\cap [f]_{  \cD_\omega}.
		\]
		Let $\phi$ is the associated extremal outer function of $[f]_{\cD_\omega}$. We have
		\[
		\mathcal{M}_\mu(\mathcal{Z}(\phi)) := \{ f \in \mathcal{D}_\mu \,:\, f = 0 \ \text{on} \ \mathcal{Z}(\phi)\ \text{$c_\mu$-q.e.} \} = \psi \mathcal{D}_{\mu_\psi},
		\]
		where $\psi$ is an outer extremal function. Since $\phi$ is outer and belongs to $\mathcal{M}_\mu(\mathcal{Z}(\phi))$, we have
		\[
		\mathcal{Z}(\psi) = \mathcal{Z}(\phi) \quad c_\mu\text{-q.e.},
		\]
		and it follows from assertion~(2) of Theorem~\ref{genspace} that
		\[
		[\psi]_{\mathcal{D}_\mu} = [\psi^2]_{\mathcal{D}_\mu}.
		\]
		Hence there exists a sequence of polynomials $(p_n)$ such that
		\[
		\|1 - p_n \psi\|_{\cD_{\mu_\psi}} = \|\psi - p_n \psi^2\|_{\cD_\mu} \longrightarrow 0,
		\quad \text{as } n \to \infty.
		\]
		Thus $\psi$ is cyclic in $\cD_{\mu_\psi}$. In particular,
		\[
		c_{\mu_\psi}(\cZ(\phi)) = c_{\mu_\psi}(\cZ(\psi)) = 0.
		\]
		Since $\mu \in \mathcal{C}$, it follows that $\phi$ is cyclic in $\mathcal{D}_{\mu_\psi}$. Therefore there exists a sequence of polynomials $(p_n)$ such that
		\[
		\|p_n \psi \phi - \psi^2\|_{\cD_\mu} 
		= \|p_n \phi - \psi\|_{\cD_{\mu_\psi}} \longrightarrow 0,
		\quad \text{as } n \to \infty.
		\]
		Consequently
		\[
		\psi^2 \in [\phi]_{\cD_\mu},
		\]
		and hence
		\[
		\cM_{\mu}(\cZ(\phi)) = [\psi]_{\cD_\mu} = [\psi^2]_{\cD_\mu} \subset [\phi]_{\cD_\mu}.
		\]
		This completes the proof.
	\end{proof}
	\begin{rem}\label{Rem}
		Similarly, in the case of measures $\nu$ defined on $\mathbb{T}$, we say that $\mathcal{C}$ satisfies property $(\mathrm{H})$ if, whenever $\nu \in \mathcal{C}$ and $\phi$ is an extremal function in $\mathcal{D}_\nu$, we have $|\phi^*|^2\, d\nu \in \mathcal{C}$.
		
		Theorem~\ref{equivalence} remains valid for these classes of measures on $\mathbb{T}$. This provides an alternative proof of \cite[Theorem~2]{EEK}. Furthermore Theorem~1 of the same paper implies Theorem~2.
	\end{rem}
	\section{Cyclicity and explicit description of invariant subspaces}
	Recall that a function $f \in \cD_\omega$ is said to be cyclic in $\cD_\omega$ if $[f]_{\cD_\omega} = \cD_\omega$, that is, if there exists a sequence of polynomials $(p_n)$ such that $p_n f \to 1$ in $\cD_\omega$. In this section, we study cyclic functions in $\cD_\omega$, beginning with some general sufficient conditions for cyclicity.
	
	By Proposition~\ref{coro1}, if $g \in \cD_\omega$ is cyclic and $f \in \cD_\omega$ satisfies $|g| \le |f|$ on $\DD$, then $f$ is also cyclic. In particular, since the polynomials are dense in $\cD_\omega$, it follows that $f$ is cyclic whenever
	\[
	\inf_{z \in \DD} |f(z)| > 0.
	\]
	As a consequence, if $f$ is outer and $|f^*| \ge c > 0$ almost everywhere on $\TT$, then $f$ is cyclic.
	

	Let $\Ee$ be a closed subset of $\TT$. Carleson~\cite{Ca} proved that if $c_m(\Ee) = 0$, then there exists a cyclic function $f \in \cD$ (not necessarily belonging to $A(\DD)$) such that $\Ee \subset \cZ(f)$. Subsequently, Brown and Cohn~\cite{BC} constructed an explicit example of a cyclic function in $\cD \cap A(\DD)$ vanishing on $\Ee$, based on a refinement of Carleson's method.
	
	Elmadani and Labghail~\cite{EL} later extended this result to the case of harmonically weighted Dirichlet spaces (see also~\cite{RY}). Using similar methods, we further extend it to the setting of Dirichlet spaces with superharmonic weights. \begin{theo}\label{Th3}
		Let $\omega$ be a positif superharmonic weight on $\DD$. Let $\Ee$ be a closed subset of $\TT$.  If $c_\omega (\Ee)=0$, then  there exists a cyclic function $f\in\cD_\omega \cap A(\DD)$ that vanishes on $\Ee$. 
	\end{theo}

	Hedenmalm and Shields~\cite{HS} established a partial result toward the Brown--Shields conjecture under the additional assumption that $\cZ(f)$ is countable (see also~\cite{RS3}). This was later refined by El-Fallah, Kellay, and Ransford~\cite{EKR2}  who replaced the countability condition with a capacity-based hypothesis. More precisely, an outer function $f \in \cD \cap A(\DD)$ is cyclic in $\cD$, if 
	\[
	\int_0^1 c(\Ee_t)\, \frac{\log\log(1/t)}{t \log(1/t)} \, dt < \infty,
	\]
	where $\Ee_t = \{\zeta \in \TT : d(\zeta, \Ee) < t\}$.
	
	An extension of this result to harmonically weighted Dirichlet spaces was obtained in~\cite{EL}. By adapting the same approach, based primarily on Theorem~\ref{ConverseStrongInequality}, we derive the following result.
	\begin{theo}\label{Th4}
		Let $\omega$ be a positif superharmonic weight on $\DD$. Let $f $ be an outer function in $ \cD_\omega \cap A(\DD) $
		and $\Ee:=\left\lbrace \zeta \in \TT : f(\zeta)=0    \right\rbrace $, if
		\begin{equation*}\label{CS-Cycl}
			\int_{0}c_\omega (\Ee_t)\frac{\log(1/t)}{t}dt <\infty,
		\end{equation*}
		then $f$ is cyclic in $\cD_\omega$
	\end{theo}
	
	\subsection{Dirichlet spaces associated with finite Borel measures on $\DD$}
	

	In this subsection, let $\mu$ be a finite positive Borel measure on $\DD$.  We shall show that $\cD_\mu$ is closely related to the Hardy space $\hH$. Recall that the balayage function $\mathcal{B}_{\mu}$ of the measure $\mu$ is defined by 
	\begin{equation*}
		\mathcal{B}_{\mu}(\zeta) = \int_\DD \frac{1 - |z|^2}{|1 - \overline{\zeta} z|^2} \, d\mu(z), \quad \zeta \in \TT.
	\end{equation*}
	Observe that $\mathcal{B}_{\mu}(\zeta)$ may be infinite. Nevertheless, an application of Fubini's theorem shows that $\mathcal{B}_{\mu} \in L^1(\TT)$. The weighted Hardy space associated with $\mathcal{B}_{\mu}$ is defined by
	\[
	\hH_{\mu} := \left\{ f \in \hH : \|f\|_{\hH_{\mu}}^2 := \int_{\TT} |f(\zeta)|^2 \mathcal{B}_{\mu}(\zeta)\, dm(\zeta) < \infty \right\}.
	\]
	Let $F_{\mu}$ denote the outer function associated with $\mathcal{B}_{\mu}$, that is,
	\[
	F_{\mu}(z) = \exp\left( \int_\TT \frac{\zeta + z}{\zeta - z} \log \frac{1}{\sqrt{\mathcal{B}_{\mu}(\zeta)}} \, dm(\zeta) \right), 
	\qquad z \in \DD.
	\]
	Then $F_{\mu} \hH = \hH_{\mu}$, and moreover,
	\[
	\|F_{\mu} f\|_{\hH_{\mu}} = \|f\|_{\hH}, 
	\qquad f \in \hH.
	\]
	The main result of this subsection is the following characterization.
	\begin{theo}
		Let $\mu$ be a positive finite Borel measure on $\DD$, and let $\cM \in \mathrm{Lat}(S_\mu)\setminus\{0\}$. Then
		$$
		\cM = \mathcal{I}_{\cM} \hH \cap \cD_\mu.
		$$
	\end{theo}
	
	\begin{proof}
		If $\mu = 0$, then $\cD_\mu = \hH$, and in this case the result is due to Beurling~\cite{Beur}.\\
		Otherwise, let $r \in (0,1)$ be such that $\mu(\DD_r) := \mu(\{ z : |z| < r \}) > 0$. Then
		\[
		\mathcal{B}_{\mu}(\zeta) \geq \frac{1 - r}{1 + r} \mu(\DD_r) > 0 
		\qquad \text{for all } \zeta \in \TT.
		\]
		As a consequence, we have $\hH_{\mu} \subset \hH$ and $\| f\|_{\hH} \lesssim \| f\|_{\hH_{\mu}}$.\\
		
		\noindent Let $\cD_\mu \cap \IIIL$ denote the Hilbert space endowed with the norm
		\[
		\| f\|^2_{\cD_\mu \cap \IIIL} = \| f\|^2_{\cD_\mu} + \|f\|^2_{\IIIL}, 
		\qquad f \in \cD_\mu \cap \IIIL.
		\]
		According to Theorem 3.3 of~\cite{BGP}, we have
		\[
		\cD_\mu \cap \IIIL = \hH_{\mu} 
		\quad \text{and} \quad 
		\| f\|^2_{\cD_\mu \cap \IIIL} = \| f\|^2_{\hH} + \| f\|^2_{\hH_{\mu}} \asymp \| f\|^2_{\hH_{\mu}}.
		\]
		
		\noindent Let $\cM \in \Lat (S_\mu)\setminus\{0\}$. 
		By Beurling's theorem, the closed invariant subspace $\cM_2 := \cM \cap L^2(\mu)$ of $\cD_\mu \cap \IIIL$ is of the form
		\[
		\cM_2 = \mathcal{I} \hH_{\mu} \subset \mathcal{I}_\cM \hH,
		\]
		where $\mathcal{I}$ denotes the inner factor associated with $\cM_2$.
		Since the inner factor of $\cM \cap \kH$ is $\mathcal{I}_\cM$ and $\cM \cap \kH \subset \cM_2$, it follows that $\mathcal{I}= \mathcal{I}_\cM$. Hence $\cM \subset \mathcal{I}_{\cM} \hH \cap \cD_\mu$. Conversely, let $f \in \kH$. Since
		\[
		\left\| \frac{f}{F_{\mu}} \right\|^2_\hH \leq \| f \|^2_\kH \, \mu(\DD) < \infty,
		\]
		we have $\mathcal{I}_{\cM} f = \mathcal{I}_{\cM} F_{\mu} \left( f / F_{\mu} \right) \in \mathcal{I}_{\cM} \hH_{\mu}$. Therefore
		\[
		\mathcal{I}_{\cM} \hH \cap \kH = \mathcal{I}_{\cM} \hH_{\mu} \cap \kH \subset \cM.
		\]
		Using the density of $(\mathcal{I}_{\cM} \hH_{\mu} \cap \cD_\mu) \cap \kH$ in $\mathcal{I}_{\cM} \hH_{\mu} \cap \cD_\mu$, we obtain
		\[
		\mathcal{I}_{\cM} \hH \cap \cD_\mu 
		= \overline{(\cM_2 \cap \cD_\mu) \cap \kH}^{\cD_\mu} \subset \cM.
		\]
		This completes the proof.
	\end{proof}
	
	It is well known that polynomials are dense in $\hH_{\mu}$. In addition, for every $f \in \hH$, one has
	\[
	\|F_{\mu} f\|_{\hH_{\mu}} = \|f\|_{\hH}.
	\]
	Hence $f$ is cyclic in $\hH_{\mu}$ if and only if it is outer.
	We further observe that if $\mu$ is finite on $\DD$, then every set of Lebesgue measure zero on $\TT$ has vanishing $c_{\mu}$-capacity. This yields the following corollary.
	\begin{coro}\label{Conj-BS-CM}
		Let $\mu$ be a positive finite Borel measure on $\DD$. Then $f \in \cD_\mu$ is cyclic in $\cD_\mu$ if and only if $f$ is an outer function.
	\end{coro}
	
	
	
	\subsection{Measures with Countable Unimodular Support}\label{Countable}
	We now consider the case of superharmonic weights with countable unimodular support. In order to prove Theorem~\ref{theo61}, We begin by presenting several preliminary lemmas.
	\begin{lem}\label{Lemma1}
		Let $f\in \cD_\mu$, and let $\Gamma=(e^{ia},e^{ib})$ be an arc of $\TT$. Then the outer function $f_\Gamma $ associated with
		\[
		|f_\Gamma (\zeta)| = 
		\begin{cases} 
			|(\zeta - e^{ia})(e^{ib} - \zeta)| |f(\zeta)|, & \zeta \in \Gamma, \\
			|(\zeta - e^{ia})(e^{ib} - \zeta)|, & \zeta \notin \Gamma,
		\end{cases}
		\]
		belongs to $\cD_\mu$.		
	\end{lem}
	The following lemma can be proved using the resolvent method. For details we refer to \cite{EEK}.	
	
	\begin{lem}\label{Lemma3}
		Let $f \in \cD_\mu \cap \kH$ be an outer function, and let $\zeta \in \TT$ such that $c_\mu(\{\zeta \})= 0$.Then\\
		
		\begin{enumerate}
			\item[A.] $\underline{\mathcal{Z}}([f]_{\cD_\mu}) \subset \underline{\mathcal{Z}}(f) \cap \supp(\mu)$.\\
			
			\item[B.] If $\underline{\mathcal{Z}}(f) \cap \supp(\mu) \subset \{\zeta\}$, then $f$ is cyclic in $\cD_\mu$.\\
			
			\item[C.] $[(z-\zeta)f]_{\cD_\mu}= [f]_{\cD_\mu}$.
		\end{enumerate}
	\end{lem}

	\begin{proof}[Proof of Theorem \ref{theo61}]
		
		Let $f$ be an outer function in $\cD_\mu$. Without loss of generality, we may assume that
		$f \in \cD_\mu \cap \kH$. Suppose that $c_\mu(\cZ(f))=0$. By Lemma \ref{Lemma3}, we have
		\[
		\underline{\mathcal Z}([f]_{\cD_\mu})
		\subset
		\underline{\mathcal Z}(f)\cap\supp(\mu).
		\]
		
		Suppose that
		$\underline{\mathcal Z}([f]_{\cD_\mu})\neq\emptyset$.
		Since
		$\underline{\mathcal Z}(f)\cap\supp(\mu)$
		is a closed countable set, there exists an arc
		$\Gamma=(e^{ia},e^{ib})$
		such that
		\[
		\Gamma\cap
		\underline{\mathcal Z}([f]_{\cD_\mu})
		=
		\{\zeta_0\}.
		\]
		
		Let $f_\Gamma$ and $f_{\TT\setminus\Gamma}$ be the outer functions defined in Lemma \ref{Lemma1}.
		By Lemma \ref{Lemma1}, both functions belong to $\cD_\mu\cap\kH$. Moreover, by
		Theorem \ref{genspace} and Lemma \ref{Lemma3}, we obtain
		\[
		[f]_{\cD_\mu}
		=
		[f_\Gamma]_{\cD_\mu}
		\cap
		[f_{\TT\setminus\Gamma}]_{\cD_\mu},
		\]
		and
		\[
		\underline{\mathcal Z}([f_{\TT\setminus\Gamma}]_{\cD_\mu})
		\subset
		\underline{\mathcal Z}([f]_{\cD_\mu})\cap\Gamma
		=
		\{\zeta_0\}.
		\]
		Since $c_\mu(\{\zeta_0\})=0$, Lemma \ref{Lemma3} implies that
		$f_{\TT\setminus\Gamma}$ is cyclic in $\cD_\mu$. Hence
		\[
		[f]_{\cD_\mu}
		=
		[f_\Gamma]_{\cD_\mu}.
		\]
		Consequently
		\[
		\underline{\mathcal Z}([f]_{\cD_\mu})
		\subset
		\underline{\mathcal Z}([f_\Gamma]_{\cD_\mu})
		\subset
		\overline{\TT\setminus\Gamma},
		\]
		which contradicts the fact that
		$\zeta_0\in
		\underline{\mathcal Z}([f]_{\cD_\mu})$.
		Therefore
		\[
		\underline{\mathcal Z}([f]_{\cD_\mu})=\emptyset,
		\]
		and hence $f$ is cyclic in $\cD_\mu$.
		
	\end{proof}
	
	
	As a consequence of Theorem~\ref{equivalence}, together with Theorem~\ref{theo61} and Remark~\ref{Rem}, we obtain the following corollary.

	\begin{coro}
		Let $\omega$ be a positive superharmonic weight on $\DD$. Let $\cM$ be a shift--invariant subspace of $\cD_\omega$ such that $\TT \cap \underline{\cZ}(\cM)$ is countable. Then 
		$$
		\cM = \mathcal{I}_{\cM}\hH \cap \{ g \in \mathcal{D}_\omega : g = 0 \ \text{on } \Ee \ \text{$c_\omega$-q.e.} \},
		$$
		where $\Ee = \{\zeta \in \TT \cap \underline{\cZ}(\cM) : \,\, c_\omega(\{\zeta\}) > 0\}$.
	\end{coro}
	
	
	\subsection{Standard weighted Dirichlet spaces $\cD_\alpha$}
	Using the new representation of Dirichlet integral provided by Theorem \ref{Function}, we extend the cyclicity results known for the classical Dirichlet space (see~\cite{EKR, EKR1}) to the spaces $\cD_\alpha$.
	\begin{theo}\label{Dalpha}
		Let $f\in\cD_\alpha$ be an outer function, and let $E:=\{\zeta\in \TT:\,\,\liminf_{z\to \zeta}|f(z)|=0\}.$ Suppose that $|E_t|=O(t^{\gamma})$ as $t\to 0$ for some $\gamma>0$, and that 
		\[
		\int_0^{\pi}\frac{dt}{t^{\alpha}|E_t|}=\infty.
		\]
		Then $f$ is cyclic in $\cD_\alpha$.
	\end{theo}
	The proof of Theorem~\ref{Dalpha} relies on several auxiliary steps. A Carleson--type estimate (Lemma~\ref{Carlesontype}) is first used to establish Proposition~\ref{llem}. We then follow the same argument as in the classical case; see \cite{EKR}, and conclude the proof of Theorem~\ref{Dalpha}.
	
	\begin{prop}\label{llem}
		Let \( \alpha \in (0, 1) \), let \( E \) be a closed subset of \( \mathbb{T} \), and let \( \sigma : [0, 2\pi] \to \mathbb{R}^+ \) be an increasing function such that \( t \mapsto \sigma(t^\gamma) \) is concave for some \( \gamma > \frac{2}{1 - \alpha} \). Let \( f_\sigma \) be the outer function satisfying \( |f_\sigma(\zeta)| = \sigma(d(\zeta, E)) \). Then
		\[
		\cD_\alpha(f_\sigma) \lesssim \int_{\TT} \sigma'(d(\zeta, E))^2 \, d(\zeta, E)^{1 + \alpha} \, dm(\zeta).
		\]
	\end{prop}
	The proof of Proposition~\ref{llem} is based on the lemma below and follows the same arguments as in~\cite{EKR}.
	\begin{lem}\label{Carlesontype}
		Let $\alpha \in (0,1)$, and let $f \in \cD_\alpha$ be an outer function. Then
		\[
		\cD_\alpha(f)
		\lesssim
		\int_{\mathbb T}
		\int_{\mathbb T}
		\frac{\bigl(|f(\zeta)|^2 - |f(\lambda)|^2\bigr)
			\log \left|\frac{f(\zeta)}{f(\lambda)}\right|}
		{|\zeta-\lambda|^{2-\alpha}}
		\, dm(\zeta)\, dm(\lambda).
		\]
	\end{lem}
	\begin{proof}
		Let $\alpha \in (0,1)$, and let $f \in \cD_\alpha$ be an outer function, then
		\[
		\cD_\alpha(f)
		=
		\int_{\mathbb D}
		\inf_{c\in\mathbb R}
		\int_{\mathbb T}
		\frac{\FF(2\log|f(\zeta)|,c)}
		{|\zeta-z|^2}
		\,dm(\zeta)\,
		(1-|z|^2)\,d\mu_\alpha(z),
		\]
		where
		\[
		d\mu_\alpha(z)=\frac{dA(z)}{(1-|z|^2)^{2-\alpha}}.
		\]
		Choosing $c=2\log|f(\lambda)|$ with $\lambda\in\mathbb T$, we obtain
		\[
		\cD_\alpha(f)
		\lesssim
		\int_0^1
		\int_{\mathbb T}
		\int_{\mathbb T}
		\frac{\FF(2\log|f(\zeta)|,2\log|f(\lambda)|)}
		{|\zeta-r\lambda|^2}
		\,dm(\zeta)\,dm(\lambda)
		\frac{dr}{(1-r)^{1-\alpha}}.
		\]
		Using the standard estimate
		\[
		|\zeta-r\lambda|^2 \asymp (1-r)^2+|\zeta-\lambda|^2,
		\]
		we deduce
		\[
		\begin{aligned}
			\cD_\alpha(f)
			&\lesssim
			\int_0^1
			\int_{\mathbb T^2}
			\frac{\FF(2\log|f(\zeta)|,2\log|f(\lambda)|)}
			{(1-r)^2+|\zeta-\lambda|^2}
			\,dm(\zeta)\,dm(\lambda)
			\frac{dr}{(1-r)^{1-\alpha}}.
		\end{aligned}
		\]
		We split the integral according to whether
		$1-r<|\zeta-\lambda|$ or $1-r\ge|\zeta-\lambda|$:
		\[
		\begin{aligned}
			\cD_\alpha(f)
			&\lesssim
			\int_0^1
			\int_{1-r<|\zeta-\lambda|}
			\frac{\FF(2\log|f(\zeta)|,2\log|f(\lambda)|)}
			{|\zeta-\lambda|^2}
			\,dm(\zeta)\,dm(\lambda)
			\frac{dr}{(1-r)^{1-\alpha}}
			\\
			&\quad
			+
			\int_0^1
			\int_{1-r\ge|\zeta-\lambda|}
			\frac{\FF(2\log|f(\zeta)|,2\log|f(\lambda)|)}
			{(1-r)^2}
			\,dm(\zeta)\,dm(\lambda)
			\frac{dr}{(1-r)^{1-\alpha}}.
		\end{aligned}
		\]
		Changing the order of integration gives
		\[
		\begin{aligned}
			\cD_\alpha(f)
			&\lesssim
			\int_{\mathbb T^2}
			\frac{\FF(2\log|f(\zeta)|,2\log|f(\lambda)|)}
			{|\zeta-\lambda|^2}
			\left(
			\int_{1-r<|\zeta-\lambda|}
			\frac{dr}{(1-r)^{1-\alpha}}
			\right)
			dm(\zeta)\,dm(\lambda)
			\\
			&\quad
			+
			\int_{\mathbb T^2}
			\FF(2\log|f(\zeta)|,2\log|f(\lambda)|)
			\left(
			\int_{1-r\ge|\zeta-\lambda|}
			\frac{dr}{(1-r)^{3-\alpha}}
			\right)
			dm(\zeta)\,dm(\lambda).
		\end{aligned}
		\]
		A direct computation shows that both inner integrals are
		$\lesssim |\zeta-\lambda|^\alpha$, hence
		\[\begin{aligned}
			\cD_\alpha(f)&
			\lesssim
			\int_{\mathbb T}
			\int_{\mathbb T}
			\frac{\FF(2\log|f(\zeta)|,2\log|f(\lambda)|)}
			{|\zeta-\lambda|^{2-\alpha}}
			\,dm(\zeta)\,dm(\lambda)\\&=
			\int_{\mathbb T}
			\int_{\mathbb T}
			\frac{\bigl(|f(\zeta)|^2-|f(\lambda)|^2\bigr)
				\log\left|\frac{f(\zeta)}{f(\lambda)}\right|}
			{|\zeta-\lambda|^{2-\alpha}}
			\,dm(\zeta)\,dm(\lambda).
		\end{aligned}
		\]
	\end{proof}
	\subsection*{Acknowledgements}
	The second named author was partially supported by the Arab Fund Foundation Fellowship Program. The Distinguished Scholar Award - File 1092. He also acknowledges the Laboratory of Analysis and Applied Mathematics (LAMA) at Gustave Eiffel University for its kind hospitality during the preparation of this paper.
	
\end{document}